\documentclass[12pt,oneside]{amsart}

\usepackage{graphicx}
\usepackage{amsfonts}
\usepackage{epsf}
\usepackage{amssymb}
\usepackage{amsmath}
\usepackage{amscd}
\usepackage{tikz}
\usepackage{pdfpages}
\usepackage{fancyhdr}
\usepackage{setspace}
\usepackage{hyperref}
\usepackage[all]{xy}
\usetikzlibrary{matrix}
\usepackage{verbatim}
\usepackage{enumerate}
\usepackage{mathrsfs}
\usepackage{pinlabel}
\usepackage{multirow}

\usepackage{color}
\usepackage[margin=3cm,marginparwidth=2.5cm]{geometry}

\theoremstyle{theorem}
\newtheorem{theorem}{Theorem}[section]
\newtheorem{proposition}[theorem]{Proposition}
\newtheorem{lemma}[theorem]{Lemma}
\newtheorem{corollary}[theorem]{Corollary}

\newtheorem{claim}[theorem]{Claim}

\theoremstyle{definition}
\newtheorem{definition}[theorem]{Definition}
\newtheorem{question}[theorem]{Question}

\theoremstyle{remark}
\newtheorem{remark}[theorem]{Remark}

\newtheorem{example}[theorem]{Example}

\newcommand{\s}{\mathfrak{s}}
\renewcommand{\t}{\mathfrak{t}}

\newcommand{\Z}{\mathbb{Z}}

\newcommand{\CP}{\mathbb{CP}}

\newcommand{\spinc}{$\text{spin}^c$~}

\DeclareMathOperator{\rk}{rank}
\DeclareMathOperator{\coker}{coker}

\newcommand{\emb}{\varepsilon}

\makeatletter
\newtheorem*{rep@theorem}{\rep@title}
\newcommand{\newreptheorem}[2]{%
\newenvironment{rep#1}[1]{%
 \def\rep@title{#2 \ref{##1}}%
 \begin{rep@theorem}}%
 {\end{rep@theorem}}}
\makeatother

\newreptheorem{theorem}{Theorem}
\newreptheorem{lemma}{Lemma}
\newreptheorem{question}{Question}
\newreptheorem{corollary}{Corollary}

\title{Embedding 3-manifolds in spin 4-manifolds}
\author{Paolo Aceto \and Marco Golla \and Kyle Larson}
\address{Alfr\'ed R\'enyi Institute of Mathematics, Budapest, Hungary}
\email{aceto.paolo@renyi.mta.hu}
\address{Matematiska institutionen, Uppsala universitet, Uppsala, Sweden}
\email{marco.golla@math.uu.se}
\address{Michigan State University, East Lansing, Michigan}
\email{larson@math.msu.edu}

\begin{document}

\rhead{\thepage}
\lhead{\author}
\thispagestyle{empty}


\raggedbottom
\pagenumbering{arabic}
\setcounter{section}{0}

\maketitle

\begin{abstract}
An invariant of orientable 3-manifolds is defined by taking the minimum $n$ such that a given 3-manifold embeds in the connected sum of $n$ copies of $S^2 \times S^2$, and we call this $n$ the embedding number of the 3-manifold. We give some general properties of this invariant, and make calculations for families of lens spaces and Brieskorn spheres. We show how to construct rational and integral homology spheres whose embedding numbers grow arbitrarily large, and which can be calculated exactly if we assume the 11/8-Conjecture. In a different direction we show that any simply connected 4-manifold can be split along a rational homology sphere into a positive definite piece and a negative definite piece.
\end{abstract}
\maketitle

\begin{section}{Introduction}\label{intro}

It is natural to ask which 3-manifolds embed in $S^4$ (or, equivalently, in $\mathbb{R}^4$). Such a 3-manifold must necessarily be orientable, and it turns out that there are different answers depending on whether one requires the embeddings to be smooth or only topologically locally flat. Freedman~\cite{F} showed that every integral homology sphere embeds topologically locally flatly in $S^4$, while there are several obstructions to a homology sphere embedding smoothly.
An integral homology sphere embedded in $S^4$ splits $S^4$ into two integral homology 4-balls,
and so any obstruction to bounding a smooth integral homology ball gives an obstruction to a smooth embedding in $S^4$. The simplest such obstruction is the Rokhlin invariant, and so any integral homology sphere with nontrivial Rokhlin invariant (for example, the Poincar\'e sphere) does not admit a smooth embedding into $S^4$. Other obstructions include the correction terms of Heegaard Floer homology, and for the case of rational homology spheres there are simpler obstructions coming from the torsion linking form and indeed the order of the first homology (it must be a square).

On the constructive side, Casson and Harer~\cite{CH} gave several infinite families of Brieskorn homology spheres that smoothly embed in $S^4$ (see~\cite{BB}). More general classes of 3-manifolds that smoothly embed in $S^4$ include those that arise as cyclic branched covers of doubly slice knots (see~\cite{Gil-Liv},~\cite{Meier},~\cite{Don}) and homology spheres obtained by surgery on ribbon links~\cite{L}. For some specific classes of 3-manifolds it is known exactly which ones smoothly embed in $S^4$, for example circle bundles over closed surfaces~\cite{Crisp-Hillman} and connected sums of lens spaces~\cite{Don}.
Budney and Burton~\cite{BB} have examined this question from the perspective of the 11-tetrahedron census of triangulated 3-manifolds.

Unfortunately a complete answer to which 3-manifolds embed in $S^4$ remains out of reach. However, the question can be generalized by asking which 3-manifolds embed in some larger class of 4-manifolds (for the case of connected sums of $\CP^2$ see~\cite{EL}). Since it is known that every orientable 3-manifold smoothly embeds in a connected sum of $S^2 \times S^2$'s, the minimum $n$ such that a given 3-manifold $Y$ \emph{smoothly} embeds in $\#_n S^2 \times S^2$
 is a well-defined invariant of $Y$, which we call the \emph{embedding number} of $Y$ and denote $\emb(Y)$. Hence the 3-manifolds that embed smoothly in $S^4$ are precisely those with embedding number equal to 0 (by convention the empty connected sum is $S^4$). Similar 3-manifold invariants, for example the surgery number (the minimal number of components of a link that admits a surgery to a given 3-manifold), are often notoriously difficult to compute. However, Kawauchi~\cite{Kaw} was able to produce infinite families of 3-manifolds whose embedding numbers grow arbitrarily large, and to compute it exactly for these manifolds (although he did not use this terminology). One drawback to his method is that it only works for 3-manifolds with non-zero $b_1$, and indeed for his examples the first Betti numbers are also unbounded. In this paper we focus on computing embedding numbers for integral and rational homology spheres.

Lens spaces are an interesting and instructive class of 3-manifolds to consider. It is known that no lens space embeds in $S^4$, but if the lens space $L(p,q)$ is punctured (that is, we remove an open ball) then it embeds in $S^4$ if and only if is $p$ is odd~\cite{Epstein,Zeeman}. For even $p$, the punctured lens space embeds in $S^2\times S^2$~\cite{EL}. Furthermore, Edmonds~\cite{Edm} showed that every lens space embeds \emph{topologically locally flatly} in $\#_4 S^2 \times S^2$. In contrast the smooth embedding numbers for lens spaces behave much differently, as we show in Section \ref{lens}. Indeed, for the family $L(n,1)$ the embedding numbers grow arbitrarily large (Proposition~\ref{bad lower bound}). We give upper and lower bounds for these embedding numbers, and give exact calculations for $n\leq19$; for even $n$ the embedding number is 1, and for odd $n$ the embedding numbers are listed in Figure~\ref{table}. As a tool we construct embeddings of $L(17,16)$ and $L(19,18)$ (and their associated canonical negative definite plumbings) into the $K3$ surface.
We also consider the question of which lens spaces have embedding number 1 (Theorem~\ref{lens space ball}), and for odd $p$ they are exactly those lens spaces that bound rational homology balls; such lens spaces were classified by Lisca~\cite{Lisca-ribbon}.

\begin{figure}
\begin{center}
\begin{tabular}{ |c|c|c|c|c|c|c|c|c|c| } 
\hline
$n$ & $3$& $5$& $7$& $9$& $11$& $13$& $15$& $17$& $19$\\
\hline
$\emb(L(n,1))$ & $2$& $4$& $6$& $8$& $10$& $10$& $8$& $6$& $4$ \\
\hline
\end{tabular}
\end{center}
\caption{Embedding numbers of $L(n,1)$, for odd $n\leq19$.}\label{table}
\end{figure}

In Section~\ref{general} we consider some general constructions and bounds. The most common technique we use to construct embeddings into $\#_n S^2 \times S^2$ is to realize the 3-manifold as surgery on an $n$-component, even-framed link (the double of the corresponding 4-manifold is $\#_n S^2\times S^2$, see Theorem \ref{all embed}), although we also use branched double cover arguments. In the other direction, most of our obstructions depend essentially on the fact that $\#_n S^2 \times S^2$ is a spin 4-manifold. Hence Rokhlin's Theorem and the 10/8-Theorem~\cite{Furuta} provide powerful tools.

Besides lens spaces, the other class of 3-manifolds we consider in depth are the Brieskorn homology spheres (in Section \ref{Brieskorn}). We give some general upper bounds on their embedding numbers, as well as some exact calculations for several infinite families where the embedding numbers are bounded. For example, each member of the family $\Sigma(2,3,6n+1)$ with $n$ odd has embedding number 10 (Proposition~\ref{Brieskorn_comp}). Work of Tange~\cite{Tange} allows us also to give families of Brieskorn spheres where the embedding numbers are unbounded, although we cannot give exact calculations.

Unfortunately, the task of giving exact calculations of arbitrarily large embedding numbers (in the case of integral or rational homology spheres) appears to be related to the gap between the 10/8-Theorem and the 11/8-Conjecture (recall the 11/8-Conjecture states that for a spin, closed 4-manifold $X$ the signature and second Betti number should be related by the inequality $b_2(X) \geq \frac{11}{8}|\sigma(X)|$, while Furuta~\cite{Furuta} proved that $b_2(X) \geq \frac{10}{8}|\sigma(X)| +2$). While the 10/8-Theorem is effective to show unboundedness of embedding numbers for many families of 3-manifolds, to give exact calculations it appears we must assume the validity of the 11/8-Conjecture (or else have counterexamples to the conjecture). In Section \ref{exact} we show how to do this by constructing integral and rational homology spheres that split connected sums of the $K3$ surface (these 4-manifolds lie on the 11/8-line) into definite pieces. In particular this method gives integral homology spheres that bound two negative definite spin 4-manifolds with different rank, answering a question of Tange \cite[Question 5.2]{Tange}. In fact our technique can be generalized using a structure theorem of Stong~\cite{Stong} to show that any simply connected 4-manifold can be decomposed into a positive definite 4-manifold and a negative definite 4-manifold (both simply connected), glued along a rational homology sphere (Theorem~\ref{definite splitting}). 

Finally, we point out that many of the techniques used in this paper are quite general, and can be applied to calculate embedding numbers for other classes of 3-manifolds than those explicitly addressed here, as well as to study embeddings of 3-manifolds into other spin 4-manifolds.

\subsection*{Acknowledgements} The authors would like to thank Bob Gompf, Ahmad Issa, D. Kotschick, Ana Lecuona, and Andr\'as Stipsicz for helpful conversations.
The first and third authors were partially supported by the ERC Advanced Grant LDTBud, and additionally the third author was partially supported by NSF grant DMS-1148490. The second author is supported by the Alice and Knut Wallenberg foundation.
\end{section}
\begin{section}{Preliminaries and general statements}\label{general}

Recall that the group ${\rm Spin}(n)$ is the double cover of $SO(n)$, which is also its universal cover as long as $n\ge 3$. A \emph{spin structure} on an $n$-manifold $M$ is a lift of the principal $SO(n)$-bundle associated to the tangent space $TM$ to a ${\rm Spin}(n)$-bundle over $M$.
A spin structure on $M$ exists if and only if $M$ is orientable and the second Stiefel--Whitney class of its tangent bundle vanishes, i.e. if and only if $w_1(M) = 0 $ and $w_2(M) = 0$; moreover, spin structures on $M$ are an affine space over $H^1(M;\mathbb{Z}/2\mathbb{Z})$.

Since every orientable 3-manifold $Y$ is parallelizable, the Stiefel--Whitney classes of its tangent bundle vanish, hence $Y$ always admits a spin structure. Moreover, if $H^1(Y;\mathbb{Z}/2\mathbb{Z}) = 0$, it is unique. This happens, for instance, when $Y$ is a rational homology sphere whose $H_1$ has odd order.

Now 4-manifolds, on the other hand, do not always admit spin structures. In fact, a closed, simply connected 4-manifold $X$ admits a spin structure if and only if it has an even intersection form.

Spin structures behave well with respect to gluing: if $(X_1, \s_1)$ and $(X_2, \s_2)$ are two spin 4-manifold with boundary $\partial X_i = Y$, then $X_1\cup (-X_2)$ admits a spin structure provided the restrictions $\s_1|_Y$ and $\s_2|_Y$ agree.
This is for free when $Y$ is a rational homology sphere whose $H_1$ has odd order.

Throughout we will assume all manifolds and maps to be smooth, and in addition we require that all manifolds be oriented. 
The following theorem is well-known (see~\cite[Section 5.7]{GS}).

\begin{theorem}\label{all embed}
Every $3$-manifold embeds in $\#_n S^2 \times S^2$ for some $n$. More precisely, every closed $3$-manifold can be realized as integral surgery on a link in $S^3$ where all the surgery coefficients are even.
If there are $n$ components in such a link, then this surgery description gives an embedding into  $\#_n S^2 \times S^2$.
\end{theorem}

\begin{proof}[Sketch of proof]
Let $Y$ be a closed 3-manifold (if a 3-manifold has boundary we can double it to obtain a closed 3-manifold, and then embed the double by the following argument).
Kaplan~\cite{Kap} gives an algorithm to realize $Y$ as integral surgery on an $n$-component link $L$ (for some $n$) where all the coefficients are even.
From this description $Y$ is realized as the boundary
of a spin 4-manifold $X$ obtained by attaching $n$ 2-handles to $B^4$ along $L$ with even framings (the intersection form is even, and hence $X$ is spin since it is simply connected). To obtain a handle decomposition for the double $\mathcal{D}X$ of $X$ we add $n$
additional 2-handles, each attached along a 0-framed meridian of a component of $L$, and then a 4-handle.
Then repeatedly sliding over these 0-framed meridians results in $n$ 0-framed Hopf pairs (see Figure \ref{f:hopf}),
which shows that $\mathcal{D}X$ is diffeomorphic to $\#_n S^2 \times S^2$ (see~\cite[Corollary 5.1.6]{GS}  for more details of the necessary handle slides).
\end{proof}

\begin{figure}[h]
\labellist
\pinlabel $0$ at -3 0
\pinlabel $0$ at 128 0
\pinlabel $0$ at 313 0
\pinlabel $0$ at 182 0
\endlabellist
\centering
\includegraphics[scale=0.8]{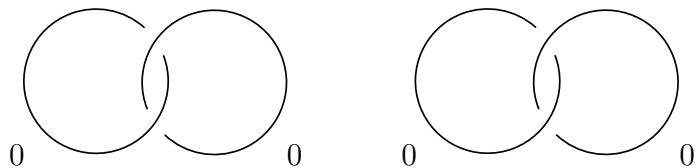}
\caption{Two 0-framed Hopf pairs and a 4-handle gives $\#_2 S^2 \times S^2$.}\label{f:hopf}
\end{figure}

Therefore the following is a well-defined invariant of 3-manifolds.

\begin{definition}
Given a 3-manifold $M$, let  $\emb(M)$ be the minimum $n$ such that $M$ embeds in $\#_n S^2 \times S^2$. Call $\emb(M)$ the \emph{embedding number} of $M$.
\end{definition}

For example, $M$ embeds in $S^4$ if and only if $\emb(M)=0$. Now we consider some general properties of this invariant.

\begin{proposition}\label{p:generalproperties}
Let $M$ and $N$ be $3$-manifolds, and let $\overline{M}$ denote $M$ with the opposite orientation. Then the embedding number satisfies the following properties:
\begin{enumerate}
\item $\emb(M) = \emb(\overline M)$;
\item $\emb(M \# N) \leq \emb(M) + \emb(N)$;
\item $\emb(M \# \overline{M}) \leq \emb(M)$.
\end{enumerate}
\end{proposition}

\begin{proof}
Point (1) is obvious, since every embedding of $M$ is also an embedding of $\overline M$.
Now we prove (2). Let $m=\emb(M)$, and $n=\emb(N)$. Then $M$ embeds in $\#_m S^2 \times S^2$ and $N$ embeds in $\#_n S^2 \times S^2$, so the disjoint union $M \sqcup N$ embeds in $\#_{m+n} S^2 \times S^2$. Since $M$ and $N$ embed disjointly in $\#_{m+n} S^2 \times S^2$, their connected sum also embeds.
Just perform ambient surgery along an embedded arc $\gamma$ in $\#_{m+n} S^2 \times S^2$ with one endpoint on $M$ and the other endpoint on $N$, such that the interior of the arc misses $M$ and $N$ (notice that in order to arrange the correct orientations we may have to change which connected component of $\#_m S^2 \times S^2\setminus M$ and $\#_n S^2 \times S^2\setminus N$ we use for the connected sum).
Therefore $\emb(M \# N) \leq m+n = \emb(M) + \emb(N)$.

Now we prove (3). Let $M^\circ$ denote $M$ with an open $B^3$ removed. If $M$ embeds in $\#_n S^2 \times S^2$, then obviously so does $M^\circ$. The boundary of a collar neighborhood $M^\circ \times I$ of $M^\circ$ is $M \# \overline{M}$, and so $M \# \overline{M}$ embeds in $\#_n S^2 \times S^2$ as well, finishing the proof.
\end{proof}

The next few results explore the relationship between small embedding numbers and bounding an integral or rational homology ball. Here and in the following, $H_*(\cdot)$ will always denote homology with integer coefficients, unless otherwise stated.

\begin{proposition}\label{Zspheres}
Let $M$ be an integral homology sphere. If $\emb(M) \leq 1$ then $M$ bounds an integral homology ball.
\end{proposition}

\begin{proof}
If $\emb(M) \leq 1$ then $M$ embeds in $S^2 \times S^2$. Let $X_1, X_2$ be the closures of the two connected components of $S^2\times S^2\setminus M$, so that $S^2 \times S^2 = X_1 \cup_M X_2$. Notice that $X_1$ and $X_2$ are spin 4-manifolds since they are codimension-0 submanifolds of the spin manifold $S^2\times S^2$.
Since $M$ is an integral homology sphere,
the Mayer--Vietoris sequence implies that $H_1(X_1)=H_1(X_2)=0$, and furthermore, we get a splitting $H := Q_{S^2 \times S^2}\cong Q_{X_1} \oplus Q_{X_2}$ using the unimodular intersection forms $Q_{X_1}$ and $Q_{X_2}$. But this implies that one of $Q_{X_1}$ or $Q_{X_2}$ is trivial
(since the forms must be even, and $H$ is the only nontrivial even, unimodular form of rank less than 8), say $Q_{X_1}$, and so $X_1$ must be an integral homology ball.
\end{proof}

Note that we do not know of any obstruction that can distinguish between an integral homology sphere embedding in $S^2 \times S^2$ and embedding in $S^4$.
Hence it is possible, although it seems unlikely, that every integral homology sphere that embeds in $S^2 \times S^2$ also embeds in $S^4$.

We now give a generalization of Proposition~\ref{Zspheres}.

\begin{theorem}\label{t:oddH1}
Let $Y$ be a rational homology sphere such that $H_1(Y)$ has odd order. If $\emb(Y) \leq 1$, then $Y$ bounds a spin rational homology ball.
\end{theorem}

\begin{proof}
Suppose $Y$ embeds in $S^2\times S^2$, splitting $S^2\times S^2$ into two spin connected components $X_1$ and $X_2$. 
Assume by contradiction that $Y$ does not bound a rational homology ball.

The Mayer--Vietoris long exact sequence reads:
\[
0 \to H_2(X_1)\oplus H_2(X_2) \to H_2(S^2\times S^2) \to H_1(Y) \to H_1(X_1)\oplus H_1(X_2) \to 0.
\]
In particular, since $H_1(Y)$ is finite, $H_2(X_1)$ and $H_2(X_2)$ are two free groups, the sum of whose ranks is $2$, and if $Y$ does not bound a rational homology ball,
then $\rk H_2(X_1) = \rk H_2(X_2) = 1$. Hence both groups are isomorphic to $\Z$.
Moreover, since $|H_1(Y)|$ is odd and it surjects onto $H_1(X_1)$ and $H_1(X_2)$, these two groups have odd order, too.

The long exact sequence of the pair for $(X_i, Y)$ reads:
\[
0 \longrightarrow H_2(X_i) \stackrel{\phi}{\longrightarrow} H_2(X_i,Y) \longrightarrow H_1(Y) \longrightarrow H_1(X_i) \longrightarrow 0,
\]
where the fact that $H_1(X_i,Y)$ vanishes follows from the surjectivity of $H_1(Y)\to H_1(X_i)$, observed above.

Now $H_2(X_i,Y)$ may have torsion, since by the universal coefficient theorem and Poincar\'e--Lefschetz duality $H_2(X_i,Y) \cong H^2(X_i) \cong H_2(X_i)\oplus H_1(X_i)$.
Let $\alpha$ be a generator of $H_2(X_i)$, and $\beta$ be the Poincar\'e dual of an element in $H^2(X_i)$ that evaluates to 1 on $\alpha$; 
then $\beta \cdot \alpha =1$ and $\phi(\alpha) = \ell \beta + t$ for some $\ell$ and some torsion element $t\in H_2(X_i,Y)$. But $\alpha \cdot \alpha = \phi(\alpha)\cdot\alpha = \ell \beta \cdot \alpha = \ell$, and so $\ell=2k$ must be an even number, since $Q_{X_i}$ is an even intersection form (because $X_i$ is spin).

Since the torsion in $H_2(X_i,Y)$ has odd order, the order of $t$ is an odd number $d$ (if $t=0$, $d=1$).
It is easy to see that the element $\bar x = dk\beta$ is \emph{not} in the image of $\phi$, while $2\bar x = \phi(d\alpha)$ is;
that is, $\bar x$ is a nonzero element in $\coker \phi$ such that $2\bar x = 0$.
But this contradicts the fact that $\coker\phi$ is a subgroup of $H_1(Y)$, which has odd order by assumption.
%
\end{proof}

This next theorem is a partial converse to Theorem \ref{t:oddH1}, and both of these theorems will be crucial in understanding which lens spaces $L(p,q)$ with odd $p$ have embedding number 1 (see Section \ref{lens}).

\begin{figure}[h]
\labellist
\pinlabel $m$ at 141 140
\endlabellist
\centering
\includegraphics[scale=1.0]{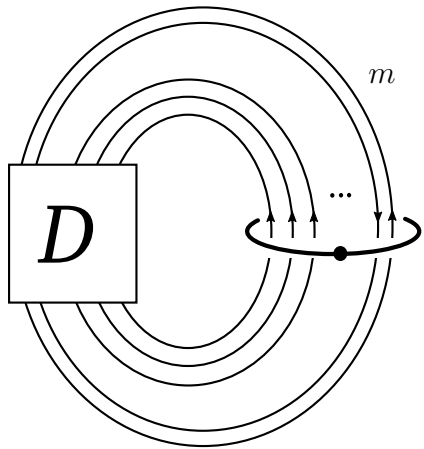}
\caption{A rational homology ball with a single 1-handle and a single 2-handle.}\label{f:rational ball}
\end{figure}

\begin{theorem}\label{t:one1-handle}
Let $Y$ be a rational homology sphere such that $H_1(Y)$ has odd order. If $Y$ bounds a rational homology ball with only a single $1$-handle and a single $2$-handle, then $\emb(Y) \leq 1$.
\end{theorem}

\begin{remark}
Note that if $Y$ is an \emph{integral} homology sphere, the argument used to prove \cite[Theorem 2.13]{BB} implies that actually $Y$ embeds in $S^4$, i.e. $\emb(Y)=0$.
\end{remark}

\begin{proof}
We need to show that $Y$ embeds in $S^2 \times S^2$.
Let $B$ be a rational homology ball with a single 1-handle and a single 2-handle, such that $\partial B= Y$. Then $B$ has a handle diagram as in Figure \ref{f:rational ball},
where there are $n$ strands running through the dotted circle, and the box labeled $D$ represents some $n$-tangle filling which results in a single attaching circle for
a 2-handle with framing $m$. Since $\partial B=Y$, the 2-handle attaching circle has an odd linking number with the dotted circle (because if we surger the 1-handle to a 0-framed 2-handle
then the intersection form of the resulting 4-manifold will present $H_1(Y)$). Note that this implies that $n$ (the total number of strands) is odd as well.
Let $\gamma$ denote the attaching circle for the 2-handle.

Now first consider the case when the framing $m$ is even. Then we attach two additional 2-handles to $B$ along 0-framed meridians of $\gamma$ and the dotted circle,
as in the left-hand side of Figure \ref{f:ballsplus}.
We can then slide $\gamma$ off the dotted circle by sliding over the 0-framed meridian of the dotted circle, and then the 2-handle attached to the meridian cancels the 1-handle.
What remains after the cancellation is a 2-handle attached to $\gamma$ with framing $m$, and another 2-handle attached to a 0-framed meridian of $\gamma$. As in the proof of Theorem \ref{all embed},
we can slide over this meridian some number of times to realize a 0-framed Hopf pair, and then cap off with a 4-handle to obtain $S^2 \times S^2$.

\begin{figure}[h]
\labellist
\pinlabel $m$ at 141 140
\pinlabel $0$ at 165 100
\pinlabel $0$ at 159 39
\pinlabel $m$ at 356 140
\pinlabel $1$ at 380 100
\pinlabel $0$ at 374 39
\endlabellist
\centering
\includegraphics[scale=0.9]{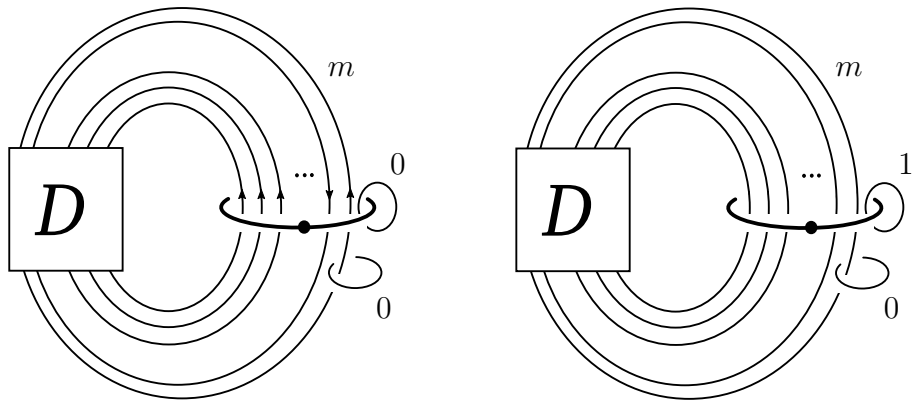}
\caption{Adding 2-handles to the rational homology ball.}\label{f:ballsplus}
\end{figure}

When $m$ is odd we again add 2-handles along meridians of $\gamma$ and the dotted circle, but this time we use framing 1 with the meridian of the dotted circle
as in the right-hand side of Figure \ref{f:ballsplus}. When we slide $\gamma$ over the 1-framed meridian we increase $m$ by one and $\gamma$ now links this meridian once
(see the left-hand side of Figure \ref{f:ballslide}).
By sliding the 1-framed meridian over the 0-framed meridian we can unlink $\gamma$ from the 1-framed meridian, and we have reduced to the starting position except $m$
has been increased by one and $\gamma$ runs one fewer time through the dotted circle (see the right-hand side of Figure \ref{f:ballslide}).

\begin{figure}[h]
\labellist
\pinlabel $m+1$ at 70 125
\pinlabel $1$ at 84 77
\pinlabel $0$ at 78 47
\pinlabel $m+1$ at 204 125
\pinlabel $1$ at 228 77
\pinlabel $0$ at 211 47
\endlabellist
\centering
\includegraphics[scale=1.0]{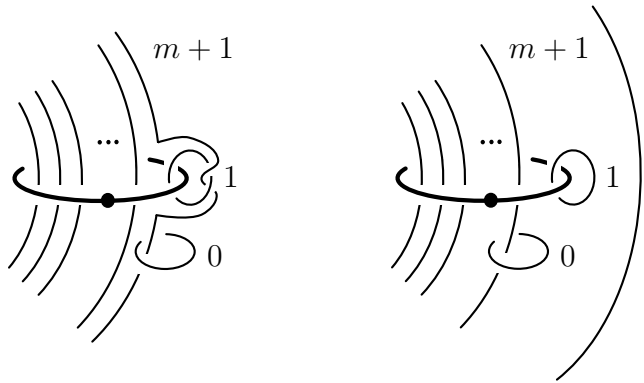}
\caption{Sliding $\gamma$ off the dotted circle.}\label{f:ballslide}
\end{figure}

We now repeat this combination of handle slides $n$ times to slide $\gamma$ completely off the dotted circle, and then the 1-framed meridian cancels the 1-handle.
What remains is a 2-handle attached to $\gamma$ with framing $m+n$, and the 2-handle attached to the 0-framed meridian. Since $n$ is odd, $m+n$ is even,
and as before we can obtain $S^2 \times S^2$ by capping off with a 4-handle.
\end{proof}

\begin{subsection}{Surgery on knots}

Let $S^3_{p/q}(K)$ denote the 3-manifold obtained by $p/q$-Dehn surgery on a knot $K$ in $S^3$. Here we prove some simple facts about embedding numbers for 3-manifolds obtained by surgery on knots.

\begin{proposition}\label{surgery prop} Let $K$ be a knot in $S^3$.
\begin{enumerate}
\item $\emb(S^3_{p/q}(K)) \geq 1$ for all $|p| >1$ and $q \neq 0$.
\item $\emb(S^3_{2n}(K)) = 1$ for all nonzero $n$.
\item $\emb(S^3_{2n+1}(K)) > 1$ if $2n+1$ is not a square.
\item $\emb(S^3_{1/{2n}}(K)) \leq 2$ for all nonzero $n$.
\end{enumerate}
\end{proposition}

\begin{proof}
\begin{enumerate}
\item We must show that $S^3_{p/q}(K)$ does not embed in $S^4$. If a rational homology 3-sphere $Y$ embeds in $S^4$, then $H_1(Y) \cong G \oplus G$ for some torsion group $G$~\cite{Gil-Liv}.
Since $H_1(S^3_{p/q}(K))\cong \Z/p\Z$, $S^3_{p/q}(K)$ does not embed in $S^4$.
\item This follows from Theorem \ref{all embed} and Part (1).
\item If $\emb(S^3_{2n+1}(K)) \leq 1$, then $S^3_{2n+1}(K)$ embeds in $S^2\times S^2$. By Theorem \ref{t:oddH1} $S^3_{2n+1}(K)$  must bound a rational homology ball,
and it is well-known that if a rational homology sphere $Y$ bounds a rational homology ball then the order of $H_1(Y)$ is a square (see, for instance, \cite[Proposition 2.2]{AG}).
\item By the reverse slam dunk move illustrated in Figure \ref{f:slamdunk} (see~\cite[Section 5.3]{GS} for a discussion of the slam dunk) with $m=2n$, we can realize $M$ as integral surgery on a
2-component link where the coefficients are even. Then by Theorem \ref{all embed} $\emb(M) \leq 2$.\qedhere
\end{enumerate}
\end{proof}

\begin{figure}[h]
\labellist
\pinlabel $1/m$ at 20 120
\pinlabel $0$ at 98 120
\pinlabel $-m$ at 130 65
\endlabellist
\centering
\includegraphics[scale=.80]{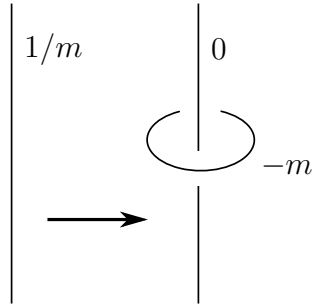}
\caption{The reverse slam dunk move.}\label{f:slamdunk}
\end{figure}

\begin{remark}
It follows from (4) above and Proposition \ref{Zspheres} that if $\emb(S^3_{1/{2m}}(K)) \neq 2$, then $S^3_{1/{2m}}(K)$ bounds an integral homology ball and so all integral homology cobordism invariants (for example, the Rokhlin invariant and the Heegaard Floer correction term) must vanish for this integral homology sphere.
\end{remark}

Proposition \ref{surgery prop}(3) suggests that most odd surgeries on a knot will have embedding number larger than 1. However, in the next example we show that this is not always the case.

\begin{example}
We can show that $\emb(S^3_9 (T_{2,3})) = 1$ by realizing $S^3_9 (T_{2,3})$ as the boundary of a rational homology ball with a single 1-handle and a single 2-handle,
and then applying Theorem \ref{t:one1-handle}. This is demonstrated in Figure \ref{f:trefoil}. We blow up to obtain the second picture, which we then think of as a 4-dimensional 2-handlebody.
Since the 0-framed 2-handle is attached along the unknot, we can surger it to a 1-handle in dotted circle notation to get the third picture.
In the third picture we see $S^3_9 (T_{2,3})$ as the boundary of the required rational homology ball.

Note that the same argument works for $S^3_{d^2}(T_{d-1,d})$ and $S^3_{d^2}(T_{d,d+1})$ for each odd $d$.
\end{example}

\begin{figure}[h]
\labellist
\pinlabel $9$ at 40 25
\pinlabel $-1$ at 368 155
\pinlabel $-1$ at 571 155
\pinlabel $0$ at 288 20
\endlabellist
\centering
\includegraphics[scale=0.7]{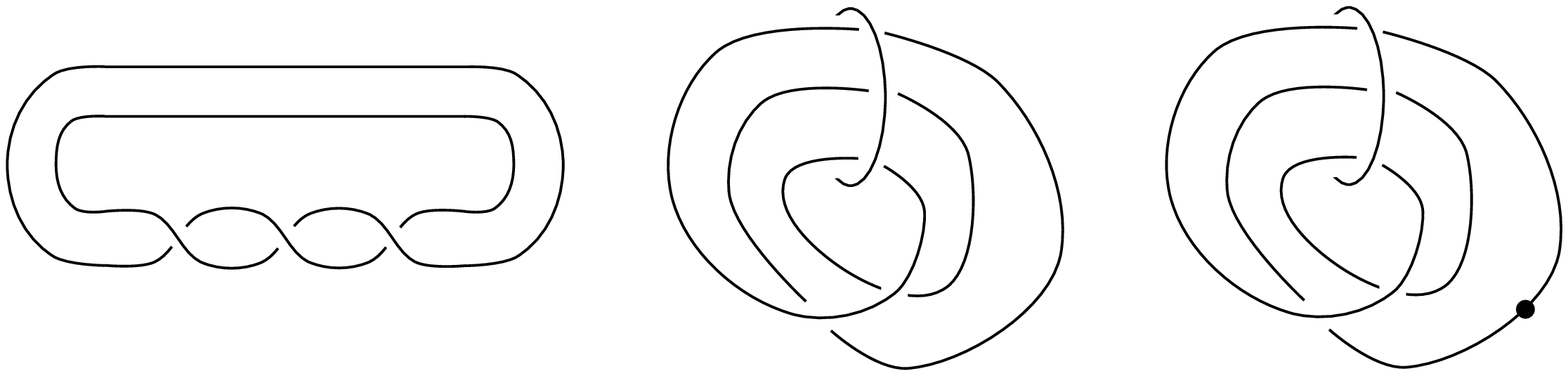}
\caption{$S^3_9 (T_{2,3})$ bounds a rational homology ball.}\label{f:trefoil}
\end{figure}

\end{subsection}

\begin{subsection}{Branched double covers} 
We finish this section by relating the embedding numbers of branched double covers to classical knot invariants.
Given a knot $K$ in $S^3$, let $\Sigma(K)$ denote the double cover of $S^3$ branched over $K$. Furthermore, let $g(K)$ denote the Seifert genus of the knot and $u(K)$ denote the unknotting number.
 
\begin{proposition}\label{DBC}
For a knot $K$ in $S^3$, let $m$ denote the minimum of $g(K)$ and $u(K)$. Then $\emb(\Sigma(K))\le 2m$.
\end{proposition}

\begin{proof}
It is a standard fact that $K$ bounds a surface in $B^4$, built with a single 0-handle and $2m$ 1-handles.
(Note that this is not the case if we replace the Seifert genus with the slice genus.)
Now Akbulut and Kirby~\cite{AK} gave an algorithm for how to build a handle decomposition of the branched double cover of $B^4$ over such a surface, and the resulting handle decomposition consists of a single 0-handle and $2m$ 2-handles, all with even framing. Since the boundary of this 4-manifold is $\Sigma(K)$, Theorem \ref{all embed} then finishes the proof.
\end{proof}

\end{subsection}
\end{section}
\begin{section}{Brieskorn spheres}\label{Brieskorn}

We now consider the embedding numbers of a specific class of 3-manifolds. Recall that the Seifert fibered manifold $\Sigma(p,q,r)$ is the boundary of the Milnor fiber $M_c(p,q,r)$;
this is a spin 4-manifold that can be constructed by taking the $p$-fold cover of $B^4$ branched over the pushed-in Seifert surface of minimal genus of the $T_{q,r}$ torus link.
Furthermore, $M_c(p,q,r)$ admits a handle decomposition with one 0-handle and $(p-1)(q-1)(r-1)$ 2-handles,
all with even framing (see~\cite{AK} and \cite[Section 6.3]{GS}). Therefore doubling $M_c(p,q,r)$ results in $\#_{(p-1)(q-1)(r-1)} S^2 \times S^2$,
and we get the following upper bound for the embedding numbers of these Seifert fibered manifolds.

\begin{proposition}\label{SFSbound}
For the Seifert fibered manifold $\Sigma (p,q,r)$, we have $\emb (\Sigma (p,q,r)) \leq (p-1)(q-1)(r-1)$.
\end{proposition}

Note that if $p$, $q$, and $r$ are relatively prime then $\Sigma (p,q,r)$ is a Brieskorn homology sphere. While Proposition \ref{SFSbound} gives an upper bound for the embedding numbers of Brieskorn spheres,
in many cases we can improve on this bound or even give an exact computation.


\begin{proposition}
If relatively prime $p$, $q$, and $r$ are all odd, have absolute value greater than $1$, and satisfy $pq+pr+qr=  -1$, then $\emb (\Sigma (|p|, |q|, |r|)) = 2$.
\end{proposition}

\begin{proof}
If $\emb (\Sigma (|p|, |q|, |r|)) < 2$, then by Proposition \ref{Zspheres}, $\Sigma (|p|, |q|, |r|)$ bounds an integral homology ball.
However, Fintushel and Stern \cite[Theorem 10.7]{FS1} showed that these manifolds never bound integral homology balls. 
Therefore $\emb (\Sigma (|p|, |q|, |r|)) \geq 2$.
Now $\Sigma (|p|, |q|, |r|)$ admits a surgery diagram with a 0-framed unknot and three meridians with framings $\pm p$, $\pm q$, and $\pm r$ (see \cite[Section 1.1.4]{Sav}).
Sliding two meridians over the third allows us to slam dunk (see Figure \ref{f:slamdunk}) the 0-framed unknot against the third meridian; this eliminates the 0-framed unknot and turns the third meridian into an $\infty$-framed curve, which can also be removed from the diagram. The result is a surgery diagram with two even-framed components.

This shows that $\emb (\Sigma (|p|, |q|, |r|)) = 2$.
\end{proof}

\begin{example}
For odd $p$, the family $\Sigma (p-2, p, (p^2-2p-1)/2)$ of Brieskorn spheres satisfy $\emb (\Sigma (p-2, p, (p^2-2p-1)/2)) = 2$.
\end{example}

\begin{proposition}\label{Poincare}
For the Poincar\'e sphere $\Sigma(2,3,5)$, we have $\emb(\Sigma(2,3,5))=8$.
\end{proposition}

\begin{proof}
Since $\Sigma(2,3,5)$ is the boundary of the $E_8$ plumbing, we immediately have $\emb(\Sigma(2,3,5))\leq 8$ and that the Rokhlin invariant $\mu(\Sigma(2,3,5))$ is nonzero.
Now assume that $\Sigma(2,3,5)$ embeds in $\#_m S^2 \times S^2$ for $m<8$, splitting $\#_m S^2 \times S^2$ into two spin pieces $U$ and $V$.
Then by the classification of indefinite, unimodular even forms and the fact that there are no definite, unimodular even forms of rank less than 8, 
both of the intersection forms $Q_U$ and $Q_V$ must have signature 0.
This contradicts the fact that $\Sigma(2,3,5)$ has nontrivial Rokhlin invariant, so $\emb(\Sigma(2,3,5))=8$.
\end{proof}

Note that this proof actually shows that any integral homology sphere with nontrivial Rokhlin invariant has $\emb(\Sigma(2,3,5))\geq8$.

The Poincar\'e sphere is one of a collection of Brieskorn spheres obtained by surgery on torus knots, namely $\Sigma(p,q,pqn \pm 1) = S^3_{-1/n}(T_{p,\pm q})$ \cite[Example 1.2]{Sav}.
Note that when $n$ is even the embedding numbers are less than or equal to $2$ by Proposition \ref{surgery prop}(4). When $n$ is odd, the situation is more difficult.

\begin{proposition}\label{Brieskorn_comp}
For any \emph{odd} integer $n>0$ we have:
\begin{itemize}
\item $\emb(\Sigma(2,3,6n+1)) = 10$. 
\end{itemize}
More generally, for any \emph{odd} integer $n>0$ and \emph{even} integer $p>0$, we have the following bound:
\begin{itemize}
\item $\emb(\Sigma(p,p+1,p(p+1)n + 1)) \leq (p+1)^2 + 1$.
\end{itemize}
\end{proposition}

\begin{figure}[h]
\labellist
\pinlabel $0$ at 15 5
\pinlabel $+1$ at 123 67
\pinlabel $n$ at -6 150
\pinlabel $n$ at 192 150
\pinlabel $n$ at 408 150
\pinlabel $-1$ at 350 45
\pinlabel $-(p+1)^2$ at 192 05
\pinlabel $-1$ at 565 45
\pinlabel $-1$ at 424 05
\pinlabel $1$ at 385 110
\pinlabel $1$ at 385 93
\pinlabel $1$ at 385 65
\endlabellist
\centering
\includegraphics[scale=.75]{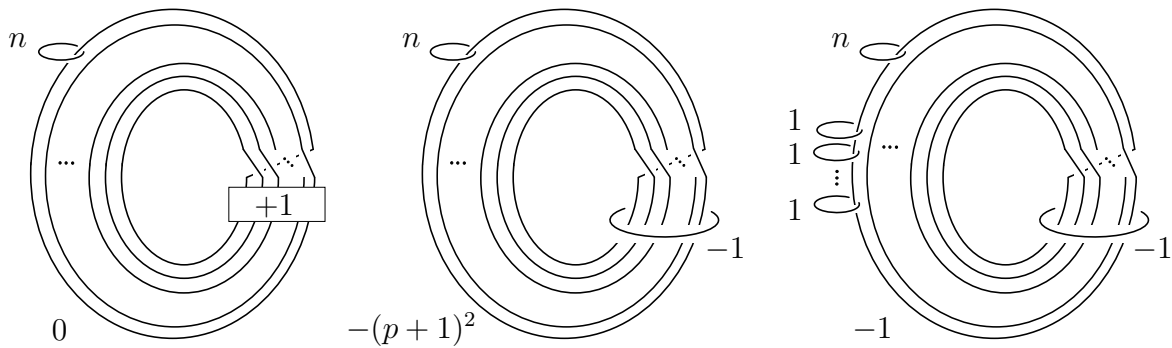}
\caption{Surgery diagrams for $\Sigma(p,p+1,p(p+1)n + 1)$.}\label{tks}
\end{figure}

\begin{proof}
We start by proving the bound $\emb(\Sigma(p,p+1,p(p+1)n + 1)) \leq (p+1)^2 + 1$. Now $\Sigma(p,p+1,p(p+1)n + 1) = S^3_{-1/n}(T_{p,p+1})$, and therefore we see a surgery diagram for these manifolds in the first picture of Figure \ref{tks},
where we have already performed a reverse slam dunk move (see Figure \ref{f:slamdunk}). Note that we draw $T_{p,p+1}$ so that there are $p+1$ strands, and hence we need to compensate for the
full right-handed twist in the $p+1$ strands by adding a $-\frac{1}{p+1}$-twist. Now the goal is to transform this surgery description into one where all the surgery coefficients are even (see~\cite{Kap} for approaches to this type of problem).
First we blow up the $p+1$ strands as in the second picture of Figure \ref{tks}. This changes the framing on our knot to $-(p+1)^2$, but now the knot is unknotted.
We blow up the knot $(p+1)^2-1$ times as in the third picture of Figure \ref{tks}. Now all the other components have odd framing, and have odd linking number with the knot (here we use the fact that $p+1$ is odd).
Hence we can blow down the knot to obtain a surgery diagram with $(p+1)^2 + 1$ even-framed components, and thus we achieve the upper bound on the embedding number by applying Theorem~\ref{all embed} as usual.

Now we consider the family $\Sigma(2,3,6n+1)$, for $n$ odd. Applying the upper bound we just obtained to the case $p=2$, we get $\emb(\Sigma(2,3,6n+1)) \leq 10$.
As in Proposition \ref{Poincare}, the fact that this each member of family has nontrivial Rokhlin invariant implies that $\emb(\Sigma(2,3,6n+1)) \geq 8$.
If for some odd $n$ we have $\emb(\Sigma(2,3,6n+1)) = 8$ or 9, then $\Sigma(2,3,6n+1)$ embeds into $\#_9 S^2 \times S^2$, splitting it into two spin pieces $U$ and $V$.
Then the intersection forms $Q_U$ and $Q_V$ are unimodular, even, and have signature $\pm8$. It then follows by the total rank that one form must be $\pm E_8$ and the other $\mp E_8 \oplus H$.
Hence $\Sigma(2,3,6n+1)$ bounds a non-standard even definite form, and we claim that this is impossible.
Indeed, Ozsv\'ath and Szab\'o computed the Heegaard Floer correction term $d(\Sigma(2,3,6n+1))=0$ \cite[Section 8.1]{OzsvathSzabo-absolutely}.
Suppose $\Sigma(2,3,6n+1)$ bounds a negative definite spin 4-manifold $W$ (if it bounds a positive definite spin 4-manifold we can reverse orientations and apply the same argument); \cite[Theorem 9.6]{OzsvathSzabo-absolutely} reads:
\[
c_1(\s)^2 + b_2(W) \le 4d(\Sigma(2,3,6n+1)) = 0.
\]
Since $W$ is even, $0$ is a characteristic vector in $H^2(W)$, and therefore it is the cohomology class of a \spinc structure $\s_0$ on $W$, and setting $\s=\s_0$ in the equation above shows $b_2(W)\le 0$.
\end{proof}


We end this section by using work of Tange to give families of Brieskorn spheres whose embedding numbers grow arbitrarily large. 

\begin{proposition}
Let $\{M_n\}$ denote any one of the following families of Brieskorn spheres (as $n$ ranges over the positive integers):
\begin{itemize}
\item $\Sigma(4n-2,4n-1,8n-3)$,
\item $\Sigma(4n-1,4n,8n-1)$,
\item $\Sigma(4n-2,4n-1,8n^2-4n+1)$, or
\item $\Sigma(4n-1,4n,8n^2-1)$.
\end{itemize}
Then $\emb(M_n) \rightarrow \infty$ as $n \rightarrow \infty$.
\end{proposition}

\begin{proof}
Tange~\cite{Tange} showed that $M_n$ bounds a spin, definite 4-manifold $X_n$ with $b_2(X_n)= 8n$.
Now, by way of contradiction, suppose there is an $m >0$ such that $\emb(M_n) \leq m$ for all $n$. In particular we can choose $a > m$ such that $8a + 2m < 10a - \frac{10}{4}m +2$, 
and $M_a$ embeds in $\#_m S^2 \times S^2$,
splitting $\#_m S^2 \times S^2$ into two spin pieces $U$ and $V$, say, with $\partial V = \overline{M_a}$. Then $Z:= X_a \cup_{M_a} V$ is a closed, spin 4-manifold,
with $$b_2(Z) \leq 8a + 2m < 10a - \frac{10}{4}m +2 = \frac{10}{8}|8a-2m| + 2 \leq \frac{10}{8}|\sigma(Z)|+2.$$
But this contradicts the 10/8-Theorem~\cite{Furuta}.
\end{proof}

\end{section}
\begin{section}{Lens spaces}\label{lens}

In this section we study the embedding numbers of lens spaces.
We give partial results for lens spaces with small embedding number, and we also study the family $L(n,1)$.
In what follows, $L(p,q)$ will always denote the 3-manifold obtained as $-p/q$-surgery along the unknot, and we will assume $p>1$ (i.e. $L(p,q)$ is not the 3-sphere nor $S^1 \times S^2$) and $\gcd(p,q)=1$.
Recall that $L(p,q)$ is the double cover of $S^3$ branched over a 2-bridge link, which we denote by $K(p,q)$; namely, $L(p,q) = \Sigma(K(p,q))$. Moreover, $K(p,q)$ is a knot if $p$ is odd, and a 2-component link if $p$ is even. Recall also that $L(p,q)$ is (orientation-preserving) diffeomorphic to $L(p,q')$ if and only if $qq' \equiv 1 \pmod p$, and that $\overline{L(p,q)}= L(p,p-q)$.

We start by giving an upper bound.

\begin{proposition}\label{bad bound 2}
$\emb(L(p,q)) \le p-1$.
\end{proposition}

\begin{proof}
Consider the linear plumbing $P$ associated to the continued fraction expansion $p/q = [a_1,\dots,a_n]^-$, namely:
\[
\xygraph{
!{<0cm,0cm>;<1cm,0cm>:<0cm,1cm>::}
!~-{@{-}@[|(2.5)]}
!{(0,0) }*+{\bullet}="x"
!{(1.5,0) }*+{\dots}="a1"
!{(3,0) }*+{\bullet}="a2"
!{(0,0.4) }*+{-a_1}
!{(3,0.4) }*+{-a_n}
"x"-"a1"
"a2"-"a1"
}
\] where each $a_i\geq2$.
This represent a surgery along a framed link $L$ that is a chain of unknots. Suppose $L'\subset L$ is a characteristic sublink (see \cite[Section 5.7]{GS}) with $\ell'$ components, indexed by the set $I'\subset \{1,\dots, n\}$.
For each $i\in I'$, we blow up the $i$-th component of $L$ $a_i-1$ times, and then we blow it down.
The resulting link $L''$ has even framing on each component, and presents $L(p,q)$ as the boundary of a spin 2-handlebody $W$.

We claim that $b_2(W)\le p-1$. Indeed, it is enough to count the number of components $\ell''$ of $L''$: they are $\ell'' = n+\sum_{i\in I'} (a_i-1) - \ell'$.

Since $a_i \ge 2$ for each $i$ we have:
\[
\ell'' = n - \ell' +\sum_{i\in I'} (a_i-1) = \sum_{i \not\in I'} 1 + \sum_{i \in I'} (a_i-1) \le \sum_{i=1}^n (a_i-1).
\]
We now prove by induction on $n$ that $\sum_{i=1}^n (a_i-1) \le p-1$.
This is obviously true when $n=1$, since in that case $a_1 = p$.
Suppose now $n>1$, and let $p = kq-r$ with $0<r<q$.
\[
\sum_{i=1}^n (a_i-1) = a_1-1 + \sum_{i=2}^n(a_i-1) = k-1+\sum_{i=2}^n(a_i-1) \le k-1+q-1,
\]
where the last inequality follows from the inductive step, and the observation that $q/r = [a_2,\dots,a_n]^-$.
Now $k-1+q-1 \le p-1$, since
\[
k-1+q \le p = kq-r \Longleftrightarrow r\le (k-1)(q-1),
\]
which is trivially true since $k\ge 2$ and $q-1\ge r$ by assumption.
\end{proof}

\begin{remark}
In fact, Neumann and Raymond show that, up to reversing the orientation, every lens space bounds a spin linear plumbing of spheres~\cite[Lemma 6.3]{NeumannRaymond}; furthermore, an easy induction shows that the number of spheres in the plumbing is at most $p-1$, giving an alternative proof of the proposition above.
\end{remark}



Now Proposition~\ref{surgery prop}(1) shows that $\emb(L(p,q)) \ge 1$.
Moreover, it has been observed by Rasmussen that every lens space in Lisca's list\footnote{As observed by several authors, the case $\gcd(m,k)=2$ should be included in type (1) in the definition of $\mathcal{R}$ in that paper.}~\cite{Lisca-ribbon} bounds a rational homology ball that admits a handle decomposition with one handle of each index 0, 1, and 2 (see~\cite{LecuonaBaker}). Combining this directly with Theorem~\ref{t:one1-handle} and Theorem~\ref {t:oddH1} gives the following result.

\begin{theorem}\label{lens space ball}
Every lens space $L(p,q)$ with $p$ odd has $\emb(L(p,q)) = 1$ if and only if it bounds a rational homology ball, i.e. if and only if belongs to Lisca's list.
\end{theorem}

This naturally leads to considering which lens spaces with $p$ even have embedding number 1.
For example, $L(pq\pm1, \mp q^2)$ with $pq$ odd has embedding number 1, since it is $(pq\pm1)$-surgery along $T_{p,q}$~\cite{Moser}.
More generally, the Berge conjecture provides more lens spaces for which the embedding number is 1: the classification of lens spaces that arise as surgery along knots in the 3-sphere has been solved by Greene~\cite{Greene}.

\begin{question}
Are all lens spaces with embedding number 1 either in Lisca's list or in Greene's list?
\end{question}

We begin the study of this problem by generalizing Theorem~\ref{t:oddH1} for some branched double cover rational homology spheres (and in particular for lens spaces).

\begin{proposition}\label{p:spin-link}
Let $L$ be a link in $S^3$ with $\ell$ components, whose branched double cover $\Sigma(L)$ is a rational homology sphere. If $\Sigma(L)$ bounds a spin $4$-manifold $W$, $b_2(W) \equiv \ell+1 \pmod 2$.
\end{proposition}

First of all, let us observe that the proposition does indeed generalize Theorem~\ref{t:oddH1} in the case of lens spaces.
Indeed, when $p$ is odd, $L(p,q)$ is the branched double cover of a knot (i.e. $\ell = 1$ in the proposition above), and every embedding of $L(p,q)$ in $S^2\times S^2$ splits the $4$-manifold into two connected components, each with even $b_2$ (namely, $0$ and $2$).
On the other hand, when $p$ is even, we have the following corollaries.

\begin{corollary}
If $p$ is even, $L(p,q)$ bounds no spin rational homology ball.
\end{corollary}

\begin{proof}
Indeed, when $p$ is even, $L(p,q)$ is the double cover of $S^3$ branched over the 2-component link $K(p,q)$, and therefore every spin filling of $L(p,q)$ has odd $b_2$.
\end{proof}

%

\begin{corollary}
If $p$ is even and $\emb(L(p,q)) = 1$, every embedding of $L(p,q)$ into $S^2\times S^2$ splits $S^2\times S^2$ into two spin $4$-manifolds $X_1$, $X_2$ with $b_2(X_i)=1$, $b_1(X_i) = b_3(X_i) = 0$.
\end{corollary}

\begin{proof}
As remarked above, $L(p,q)$ is the branched double cover of a $2$-component link, hence $\ell = 2$ in the proposition above, and therefore $b_2(X_i)$ is odd.
Looking at the Mayer--Vietoris sequence, we have that $b_1(X_i) = b_3(X_i) = 0$; since $b_2(S^2\times S^2) = 2$ and the $b_2(X_i)$ are odd, we have $b_2(X_1) = b_2(X_2) = 1$, hence the corollary follows.
\end{proof}

\begin{remark}
The classification of which lens spaces bound a $4$-manifold $X_1$ as in the corollary above is still an open question.
Observe also that, if we restrict our attention to lens spaces $L(p,q)$ with squarefree, even $p$, there is a further restriction on the homology of $X_1$; namely, $H_1(X_1) = H_3(X_1) = 0$, and $H_2(X_1) = \Z$.
However, this restriction is not sufficient to allow for a complete classification, either.
\end{remark}

\begin{proof}[Proof of Proposition~\ref{p:spin-link}]
By restriction $\Sigma(L)$ inherits a spin structure $\t$ from $W$.
Turaev~\cite[Section 2.2]{Turaev} has shown that to each orientation $o$ on $L$ one can associate a spin structure $\t_o$ on $\Sigma(L)$, and Donald and Owens~\cite[Proposition 3.3]{DonaldOwens} gave the following interpretation of $\t_o$.
Fix a Seifert surface for the oriented link $(L,o)$, and push it into the 4-ball, obtaining a surface $F_o$; the branched double cover $\Sigma(B^4,F_o)$ admits a spin structure $\s_{F_o}$ (the pull-back of the spin structure on $B^4$), and $\t_o$ is the restriction of $\s_{F_o}$ to $\Sigma(L)$.

Recall that $\sigma(L,o)$ is the signature of the double cover $\Sigma(B^4,F_o)$.

In the case at hand, since $\Sigma(L)$ is a rational homology sphere, $\det L\neq0$, and the Seifert form of $F_o$ is nondegenerate for every choice of orientation $o$ and every choice of Seifert surface $F_o$.
It follows that $\sigma(L,o) \equiv b_1(F_o) \pmod 2$, and since $b_1(F_o) \equiv \ell+1 \pmod 2$, $\sigma(L,o) \equiv \ell+1 \pmod 2$.

Summing up, for each spin structure on $\Sigma(L)$ we constructed a spin filling, whose signature is congruent to $\ell+1$ modulo 2.
In particular, we have a spin filling $W_\t$ of $(Y,\t)$. We can glue $W_\t$ and $-W$ along $Y$, and we obtain a closed spin $4$-manifold $X$.
Thus, since closed spin 4-manifolds have even signature, $\sigma(W_\t) - \sigma(W) = \sigma(X) \equiv 0 \pmod 2$, from which the result follows.
%
%
%
\end{proof}

\begin{subsection}{The family $L(n,1)$}

We now focus on lens spaces of the form $L(n,1)$. For convenience we work with the opposite orientation, $\overline{L(n,1)}=L(n,n-1)$, and so let $L_n$ denote the lens space $L(n,n-1)$. 
According to Proposition~\ref{surgery prop}, $L_{2n}$ has embedding number 1.
However, we can refine the notion of the embedding number by requiring that the restriction of the unique spin structure on $\#_nS^2\times S^2$ induce a given spin structure on the 3-manifold.
In particular, the embedding of $L_{2n}$ in $S^2\times S^2$ realizes the restriction of the unique spin structure on the 2-handlebody obtained by attaching a 2-handle to $B^4$ along the unknot, with framing $2n$.
We can ask what happens with the other spin structure on $L_{2n}$, which is the restriction of the spin structure on the plumbing $P_{2n}$ of a chain of $2n-1$ spheres with Euler number $-2$.
Observe also that $L_{2n+1}$ is the boundary of the plumbing $P_{2n+1}$, and restricting the spin structure on $P_{2n+1}$ gives the unique spin structure on $L_{2n+1}$.

For the remainder of this discussion we consider $L_n$, implicitly equipped with the spin structure described above. 
Proposition~\ref{bad bound 2} states that $\emb(L_n) \leq n-1$. Indeed, we can see this directly since $L_n$ is the boundary of the linear plumbing $P_n$ (the induced surgery diagram is a chain of $n-1$ unknotted components, all with framing $-2$).

We will see below that this upper bound can be improved upon in many cases. But first we use the 10/8-Theorem to show that the embedding numbers of these lens spaces grow arbitrarily large. Recall from the introduction that Edmonds~\cite{Edm} proved that every lens space embeds topologically locally flatly in $\#_4 S^2 \times S^2$, and hence we see that these embedding numbers reflect the sharp contrast between the smooth and topological categories of 4-manifolds.

In what follows, we will repeatedly use Novikov additivity. More precisely, whenever we have a splitting of $\#_nS^2\times S^2$ into $X$ and $Y$ along a 3-manifold, we have $\sigma(X) + \sigma(Y) = \sigma(\#_nS^2\times S^2) = 0$; that is, $\sigma(X) = -\sigma(Y)$.

\begin{proposition}\label{bad lower bound}
$\emb(L_n) \geq \frac19 (n+7)$.
Additionally, if the $11/8$-Conjecture is true, then $\emb(L_n) \geq \frac{3}{19}(n-1)$.
\end{proposition}

\begin{proof}
As noted above, $L_n = \partial P_n$, where the latter is a spin $4$-manifold with $\sigma(P_n) = 1-n$.

Suppose $\emb(L_n) = m$, and fix an embedding of $L_n$ into $\#_{m}S^2\times S^2$. Note we can assume that $m\leq n-1$ by Proposition \ref{bad bound 2}.
Let $Z$ be the connected component of the complement of this embedding such that $\partial Z = -L_n$.
Consider the closed, spin $4$-manifold $X := P_n \cup_{L_n} Z$: since $L_n$ is a rational homology sphere, $b_2(X) = n-1+b_2(Z)$ and $\sigma(X) = 1-n+\sigma(Z)$.
Observe that, by definition, $\sigma(Z) \le b_2(Z)$.

By the 10/8-Theorem~\cite{Furuta},
\begin{equation}\label{e:furuta}
n-1 + b_2(Z) = b_2(X) \ge \frac54|\sigma(X)| + 2 = \frac54|1-n+\sigma(Z)| + 2= \frac54(n-1-\sigma(Z))+2,
\end{equation}
and in particular $7-4n-4b_2(Z) \le 5\sigma(Z)-5n \le 5b_2(Z)-5n$, from which we obtain $b_2(Z) \geq \frac19n +\frac79$.

Let $Z' = \#_{m}S^2\times S^2 \setminus Z$. By Novikov additivity, $\sigma(Z') = -\sigma(Z)$, and, by gluing $-P_n$ onto $Z'$, the same manipulations as above give the inequality
\[
m = \frac12 b_2(\#_{m}S^2\times S^2) = \frac12(b_2(Z) + b_2(Z')) \geq \frac19n+\frac79,
\]
as desired.

Assuming the 11/8-Conjecture, instead of~\eqref{e:furuta} we have
\[
n-1 + b_2(Z) = b_2(X) \ge \frac{11}{8}|\sigma(X)| = \frac{11}{8}|1-n+\sigma(Z)|,
\]
from which one readily obtains $3n-8b_2(Z) \le 11\sigma(Z)+3$, from which the desired bound follows as before.
\end{proof}

Now we show two relations between $\emb(L_n)$: one is a form of subadditivity, and the other asserts that $\emb(L_{n-1})$ gives tight restrictions on $\emb(L_n)$; in particular, the values can differ by at most 1.

\begin{theorem}\label{steps}
We have:
\begin{enumerate}
\item $\emb(L_{m+n}) \le \emb(L_m)+\emb(L_n)+1$;
\item $\emb(L_{n-1})-1 \leq \emb(L_{n}) \leq \emb(L_{n-1}) + 1$.
\end{enumerate}
\end{theorem}

We first pause and observe a consequence of (1) in the theorem above.

\begin{corollary}
The sequence $\emb(L_n)/n$ converges to a limit $\emb_L\in \mathbb{R}$.
Moreover, $\emb_L\ge \frac19$, and, if the $11/8$-Conjecture holds, $\emb_L \ge \frac3{19}$.
\end{corollary}

\begin{proof}
The sequence $a_n = \emb(L_n)+1$ is subadditive by Theorem~\ref{steps}(1), above.
Therefore, by Fekete's lemma~\cite{Fekete}, we have $\inf_n \frac{a_n}n = \lim_n \frac{a_n}n$.

Since $a_n = \emb(L_n) + 1$, we have $\lim_n \frac{\emb(L_n)}n= \lim_n \frac{a_n}n$, and in particular the sequence ${\emb(L_n)}/n$ is convergent.

The inequalities follow directly from Proposition~\ref{bad lower bound} above.
\end{proof}

\begin{remark}
Steven Sivek pointed out to the authors that we can obtain upper bounds for $\emb_L$ as follows. Suppose $\emb(L_n) =e$ for some $n$. Then by induction one can show that $\emb(L_{2^k n})\leq2^k(e+1)-1$, and hence $\emb_L \leq \frac{1}{n}(e+1)$. For example, below we will show that $\emb(L_{19})=4$, and so this implies that $\emb_L \leq \frac{5}{19}$.
\end{remark}

\begin{proof}[Proof of Theorem~\ref{steps}]
\begin{enumerate}
\item By Proposition~\ref{p:generalproperties}, it is enough to prove that $\emb(L_{m+n}) \le \emb(L_m\#L_n)+1$.
Suppose therefore that $\emb(L_m\# L_n) = a$, i.e. that $\#_a S^2\times S^2 = X \cup_{L_m\# L_n} Y$.
Consider the cobordism $W_0$ in Figure~\ref{cobordism0};
this is obtained from $(L_m\#L_n)\times I$ (represented by the disjoint plumbings of $m-1$ and $n-1$ components with framings $-2$ in brackets) by adding a $-2$-framed 2-handle $h$ joining two ends of $P_m$ and $P_n$, and a 2-handle $b$ attached along a meridian of the attaching curve of $h$ with framing 0.
The upper boundary of $W_0$ is still $L_m\# L_n$. Notice that $W_0$ contains $L_{m+n}$, since it contains the boundary of the plumbing of a chain of $-2$-framed unknots of length $m+n-1$.
In particular, the closed 4-manifold $Z_0 = X\cup_{L_m\# L_n} W_0\cup_{L_m\# L_n} Y$ contains a copy of $L_{m+n}$. We will show in  Lemma~\ref{l:trivialcob} that $Z_0$ is diffeomorphic to $\#_{a+1} S^2\times S^2$, and hence we obtain $\emb(L_{m+n}) \le a+1$, as desired.

\item Suppose that $\emb(L_n)=a$, so that $\#_a S^2 \times S^2 = X \cup_{L_{n}} Y$. Consider the cobordisms $W_1$ in Figure \ref{cobordism} and $W_2$ in Figure \ref{cobordism2}, each consisting of two 2-handles (labelled $h,b$) added to $L_n \times I$. Notice that in both cases the upper boundary of the cobordism is still $L_n$. Hence we can form new spin 4-manifolds $Z_i := X \cup W_i \cup Y$. Now observe that $L_{n-1}$ is embedded in $W_1$ and $L_{n+1}$ is embedded in $W_2$ as middle levels. In both cases, just take the surgery description given by the plumbing of $-2$ spheres and add the surgery curve corresponding to $h$. For $W_2$ the claim is immediate, and for $W_1$ we must slam dunk the 0-framed curve $h$. Now to complete the proof we only need to show that each $Z_i$ is diffeomorphic to $\#_{a+1} S^2 \times S^2$ (to get the right-hand inequality we reindex), and we do this in the next lemma.\qedhere
\end{enumerate}
\end{proof}

\begin{figure}[h]
\labellist
\pinlabel $\langle-2\rangle$ at 35 -5
\pinlabel $\langle-2\rangle$ at 390 -5
\pinlabel $\hphantom{\langle-2\rangle}h$ at 180 -5
\pinlabel $b$ at 188 90
\pinlabel $0$ at 228 90
\pinlabel $\hphantom{\langle-2\rangle}-2$ at 208 -5
\endlabellist
\centering
\includegraphics[scale=.95]{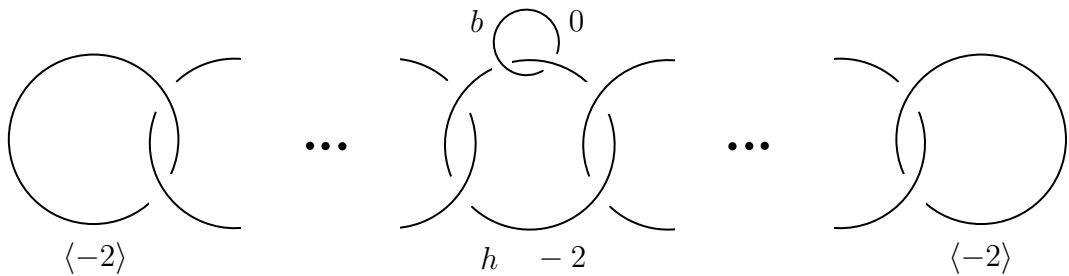}
\caption{The cobordism $W_0$.}\label{cobordism0}
\end{figure}

\begin{figure}[h]
\labellist
\pinlabel $\langle-2\rangle$ at 35 -5
\pinlabel $\langle-2\rangle$ at 92 -5
\pinlabel $\langle-2\rangle$ at 148 -5
\pinlabel $\langle-2\rangle$ at 329 -5
\pinlabel $\langle-2\rangle$ at 386 -5
\pinlabel $h$ at 12 79
\pinlabel $b$ at 12 109
\pinlabel $0$ at 54 79
\pinlabel $0$ at 52 109
\endlabellist
\centering
\includegraphics[scale=.95]{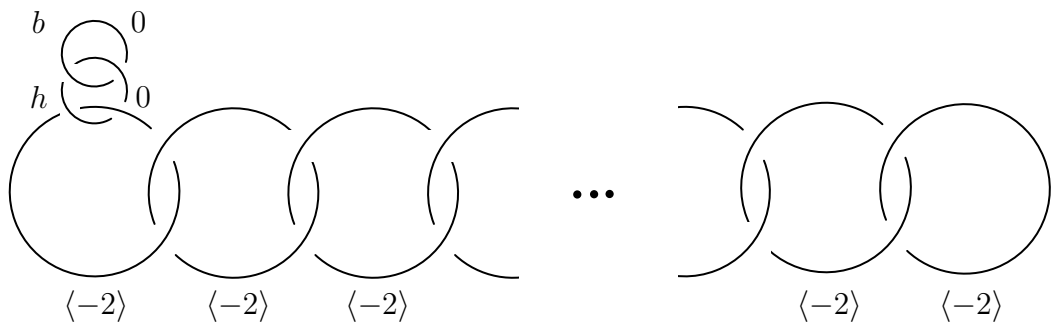}
\caption{The cobordism $W_1$.}\label{cobordism}
\end{figure}

\begin{figure}[h]
\labellist
\pinlabel $\langle-2\rangle$ at 35 -5
\pinlabel $\langle-2\rangle$ at 92 -5
\pinlabel $\langle-2\rangle$ at 148 -5
\pinlabel $\langle-2\rangle$ at 329 -5
\pinlabel $\langle-2\rangle$ at 386 -5
\pinlabel $h$ at 12 79
\pinlabel $b$ at 12 109
\pinlabel $-2$ at 57 79
\pinlabel $0$ at 54 109
\endlabellist
\centering
\includegraphics[scale=.95]{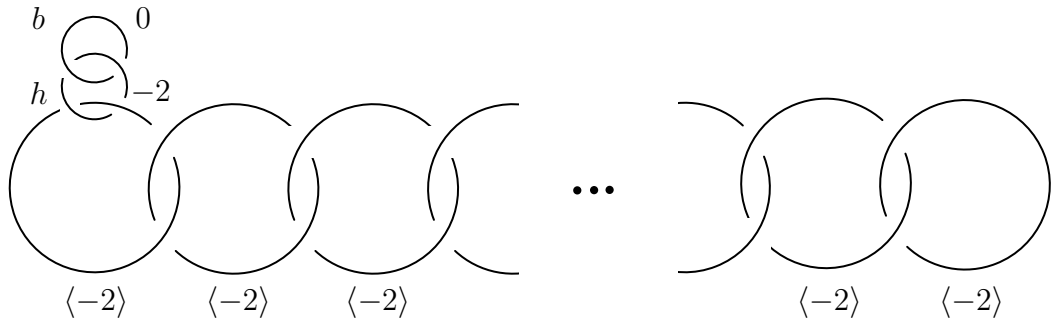}
\caption{The cobordism $W_2$.}\label{cobordism2}
\end{figure}

\begin{lemma}\label{l:trivialcob}
In each case above, $Z_i$ is diffeomorphic to $\#_{a+1} S^2 \times S^2$, for $i=0,1,2$.
\end{lemma}

\begin{proof}
Let $L$ denote either $L_m\#L_n$ or $L_n$, depending on the respective case.
In any of the three cases we can surger $b$ in $W_i$ to a 1-handle in dotted circle notation, and the resulting cobordism is the trivial cobordism $L\times I$ since $h$ will cancel the resulting 1-handle. Therefore we obtain $W_i$ by reversing this operation, i.e. performing surgery on a copy of $S^1 \times D^3$ in $L \times I$. Now consider this operation occurring inside the closed manifold $X \cup (L \times I) \cup Y =\#_a S^2\times S^2$ (where $X$ and $Y$ come from the proof of Theorem~\ref{steps} above, and depend on which case we are in). Since this manifold is simply connected, the surgery curve $S^1 \times \{0\}$ is null-homotopic
and the effect of the surgery is to produce a connected summand with one of the two $S^2$ bundles over $S^2$~\cite[Proposition 5.2.3]{GS}.
We now claim that $W_i$ is spin: in fact, $W_i$ embeds in the spin $2$-handlebody obtained by removing all brackets in Figures~\ref{cobordism0}--\ref{cobordism2} above. It follows that we must be taking a connected sum with $S^2 \times S^2$, and we get that $Z_i= X \cup W_i \cup Y$ is diffeomorphic to $\#_{a+1} S^2 \times S^2$, as claimed.
\end{proof}

We end this section with a lengthy discussion of the embedding numbers for $L_n$ with $n\leq 19$. We again emphasize that in this section we are abusing notation by using $\emb(L_{2n})$ to denote the embedding number of $L_{2n}$ with respect to the spin structure that extends over $P_{2n}$.


\begin{theorem}\label{t:smallLn}
The embedding numbers for $L_n$ with $n\leq 19$ are given by the following table.
%
\begin{center}
\begin{tabular}{ |c|c|c|c|c|c|c|c|c|c|c|c|c|c|c|c|c|c|c| } 
\hline
$n$ & $2$ & $3$ & $4$ & $5$ & $6$ & $7$ & $8$ & $9$ & $10$ & $11$ & $12$ & $13$ & $14$ & $15$ & $16$ & $17$ & $18$ & $19$\\
\hline
$\emb(L_n)$ & $1$ & $2$ & $3$ & $4$ & $5$ & $6$ & $7$ & $8$ & $9$ & $10$ & $11$ & $10$ & $9$ & $8$ & $7$ & $6$ & $5$ & $4$\\
\hline
\end{tabular}
\end{center}
%
\end{theorem}


We will make several of these computations in detail: in particular, we will compute $\emb(L_2)$, $\emb(L_{12})$, and $\emb(L_{19})$, and these will suffice to prove the theorem.
We will also give a more explicit computation of $\emb(L_{17})$.

Observe that for these values of $n$, Theorem~\ref{bad lower bound} is insufficient to give the necessary lower bounds. From Proposition~\ref{bad bound 2} it follows that $\emb(L_n) \leq n-1$, and for $n=3, \cdots, 11$ we claim that this cannot be improved upon.

%

\begin{claim}
$\emb(L_2) = 1$.
\end{claim}

\begin{proof}
Indeed, $L_2$ is understood to have the spin structure induced as the boundary of $P_2$, whose double is $S^2\times S^2$.
Moreover, since the order of $H_1(L_2)$ is not a square, $\emb(L_2)\ge 1$, hence $\emb(L_2) = 1$.
\end{proof}

%
%

\begin{claim}
$\emb(L_{12})=11$.
\end{claim}

\begin{proof}
Since $L_{12}$ is the boundary of $P_{12}$, by doubling we obtain $\emb(L_{12})\le 11$.
If $\emb(L_{12}) < 11$, then $L_{12}$ embeds in $\#_{10} S^2 \times S^2$, splitting it into two spin pieces $U$ and $V$, say, with $\partial U =L_{12}$ and $\partial V = \overline{L_{12}}$. Since the Rokhlin invariant associated to this spin structure satisfies $\mu(L_{12}) = -11 \pmod{16}$, and $b_2(\#_{10} S^2 \times S^2) = 20$, we must have that $\sigma(U)=5$ and $\sigma(V)=-5$. Furthermore, at least one of $U$ or $V$ must have $b_2 \leq 10$, and we can glue this manifold to either $P_{12}$ or $\overline{P_{12}}$ to get a spin closed 4-manifold $X$ with $b_2(X) \leq 21$ and $|\sigma(X)|=16$. But then
\[
b_2(X) \leq 21 < \frac{10}{8}(16) +2= \frac{10}{8}|\sigma(X)| +2,
\]
which contradicts the 10/8-Theorem.
\end{proof}

Before considering $\emb(L_{19})$ we will first prove directly that $\emb(L_{17}) =6$, since the strategy is the same as for $L_{19}$ and the argument in this case is easier. In both cases we claim that the plumbings $P_{17}$ and $P_{19}$ embed in the $K3$ surface. For $L_{17}$ we show directly that the complement $K3\setminus P_{17}$ consists of a single 0-handle and six 2-handles (necessarily with even framing, since the $K3$ surface is spin), and so the standard argument shows that $\emb(L_{17})\leq6$. However, we were unable to carry out fully the analogous argument for $L_{19}$. Instead we embed $P_{19}$ in a \emph{homotopy} $K3$ such that the complement admits a handle decomposition with a single 0-handle and four even-framed 2-handles, giving $\emb(L_{19})\leq4$.

The following fact is related to, though not directly relevant for, computing $\emb(L_n)$.

\begin{proposition}\label{sextic cover}
The plumbings $P_n$ embed in the standard smooth $K3$ surface for every $n\le 19$.
\end{proposition}

\begin{proof}
It clearly suffices to prove the statement for $n=19$.
To this end, recall that there exists a complex sextic $C$ in $\CP^2$ with a unique singular point, whose link is the torus knot $T_{2,19}$ \cite[Theorem 5.8]{sextics}.
We can smooth out the singularity of $C$ by taking a small perturbation, thus obtaining a complex curve $C'$.
In a neighborhood of the singular point, this has the effect of replacing the singularity with its Milnor fiber $F$, that is isotopic to a push-off of the minimal-genus Seifert surface of $T_{2,19}$ pushed into $B^4$.
By taking the double cover of $\CP^2$ branched over $C'$, we obtain a complex $K3$ surface, that contains the double cover of $B^4$ branched over $F$, and this is known to be a plumbing of $-2$-spheres of length $19$.
Therefore, $P_{19}$ embeds in a $K3$.
\end{proof}

Now for $L_{17}$ we can see an embedding of $P_{17}$ explicitly in a handle diagram for the $K3$ surface.

\begin{claim}
$\emb(L_{17})=6$. Indeed, $P_{17}$ embeds in the $K3$ surface such that the complement admits a handle decomposition with a single $0$-handle and six even-framed $2$-handles.
\end{claim}

\begin{figure}[h]
\labellist
\pinlabel $-2$ at 6 28
\pinlabel $0$ at 47 164
\pinlabel $-1$ at 106 5
\pinlabel $0$ at 141 197
\pinlabel $2$ at 152 98
\pinlabel $3$ at 277 98
\pinlabel $1$ at 277 180
\pinlabel $3$ at 221 98
\pinlabel $1$ at 221 180
\pinlabel $3$ at 333 98
\pinlabel $1$ at 333 180
\pinlabel $3$ at 388 98
\pinlabel $1$ at 388 180
\pinlabel $-1$ at 232 50
\pinlabel $-1$ at 259 50
\pinlabel $-1$ at 296 50
\pinlabel $15$ at 269 -10

\endlabellist
\centering
\includegraphics[scale=0.85]{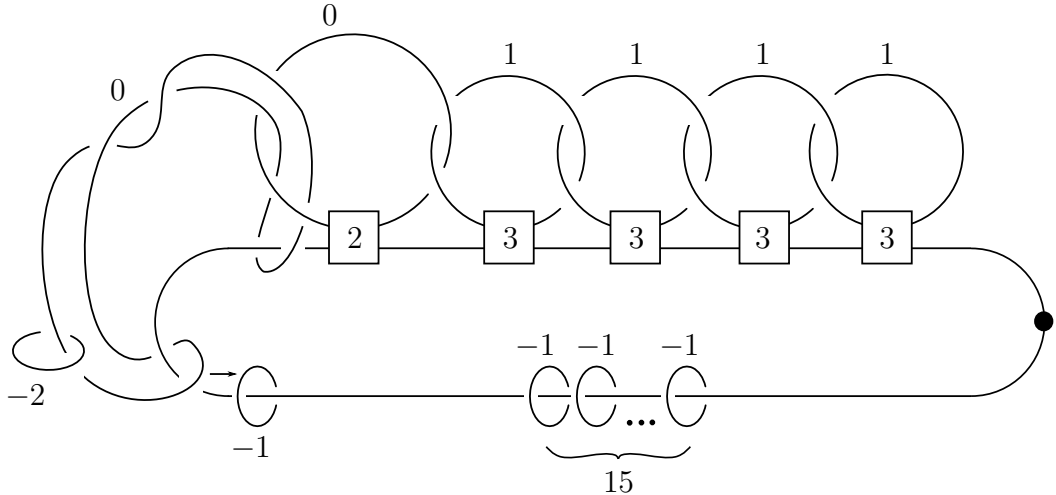}
\caption{A handle diagram for the $K3$ surface.}\label{K3}
\end{figure}

\begin{proof}
In \cite[Section 8.3]{GS} it is shown that Figure~\ref{K3} (plus a 4-handle) is a handle diagram for the $K3$ surface. Notice that in total there are sixteen $-1$-framed meridians. These can be slid, one over the other, as in Figure~\ref{chain}, to form a chain of fifteen $-2$-framed unknots, attached at the end to a $-1$-framed meridian that cancels the 1-handle. Hence we immediately see an embedding of $P_{16}$ into the $K3$ surface. To obtain $P_{17}$ we must work a little harder. If in Figure~\ref{K3} we slide the left-most 0-framed 2-handle over the left-most $-1$-framed meridian as indicated by the arrow, then after isotopy the 0-framed 2-handle becomes a $-1$-framed meridian of the dotted circle, itself with a $-2$-framed meridian (and of course linking the the meridian it slid over). Now we may begin the process of sliding the meridians to create the $-2$ chain, but we gained an extra $-2$-framed 2-handle (the original $-2$-framed 2-handle on the left of the diagram). After these slides we will see $P_{17}$ embedded in the $K3$ surface, linked with two $-1$-framed meridians. Either of these meridians cancels the 1-handle after sliding the remaining 2-handles off. The result will be a handle diagram for the $K3$ surface composed of twenty-two 2-handles, a single 0- and 4-handle, and with $P_{17}$ as a sub-handlebody. Then if we take the 2-handles not in $P_{17}$ and the 4-handle, and turn this handlebody upside down, we get a handlebody with a single 0-handle, and six even-framed 2-handles (necessarily even-framed, since the $K3$ surface is spin) with $L_{17}$ as the boundary. Hence $\emb(L_{17})\leq 6$.

Now we finish the claim by showing $L_{17}$ does not embed in $\#_5 S^2 \times S^2$. If it did, it would split $\#_5 S^2 \times S^2$ into two spin pieces $U$ and $V$. Since $\mu(L_{17})=-16 \equiv 0$ (mod 16), we have $\sigma(U)=\sigma(V)=0$. At least one of $U$ or $V$ must have $b_2 \leq5$, and gluing this manifold to either $P_{17}$ or $\overline{P_{17}}$ results in a closed spin 4-manifold $X$ with
\[
 b_2(X)\leq 21  < \frac{10}{8}(16) +2= \frac{10}{8}|\sigma(X)| +2,
\]
which contradicts the 10/8-Theorem.
\end{proof}

\begin{figure}[h]
\labellist
\pinlabel $-1$ at 22 28
\pinlabel $-1$ at 47 28
\pinlabel $-1$ at 76 28
\pinlabel $-1$ at 131 61
\pinlabel $-2$ at 132 16
\pinlabel $-1$ at 174 28
\pinlabel $-2$ at 227 30
\pinlabel $-1$ at 229 61
\pinlabel $-2$ at 227 10
\endlabellist
\centering
\includegraphics[scale=1.25]{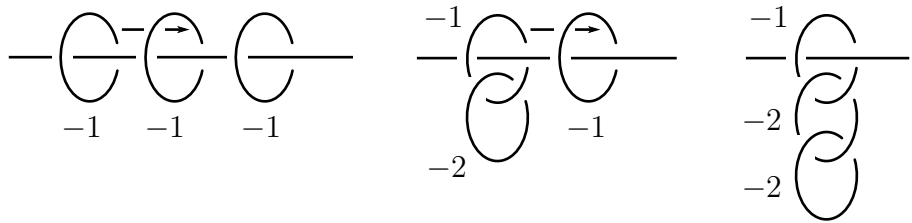}
\caption{Sliding $-1$-framed meridians.}\label{chain}
\end{figure}

Observe that our argument shows that $L_{17}$ can be obtained by even-framed surgery on a 6-component link in $S^3$. In principle one could follow through the handle calculus to realize this surgery diagram, but we have not done this.

For $L_{19}$ we were unable to prove that the complement of $P_{19}$ in the $K3$ surface (given by the embedding described in Proposition~\ref{sextic cover}) admits a handle decomposition  without 1- or 3-handles. Instead we will construct a certain \emph{smooth} sextic in $\CP^2$ whose branched double cover is a homotopy $K3$, which we denote by $\mathcal{K}$, such that $P_{19}$ embeds in $\mathcal{K}$ with a complement that admits a handle decomposition with a single 0-handle and four even-framed 2-handles.

\begin{lemma}
There is a genus-$1$ cobordism $C$ from $T_{2,19}$ to $T_{6,6}$ with no minima and no maxima.
\end{lemma}

\begin{proof}
Consider the group $B_6$ of braids on six strands; we let $\sigma_1,\dots,\sigma_5$ denote the five standard generators, and $\tau = \sigma_1\sigma_2\dots\sigma_5$ be their product; the full twist on six strands is then $\tau^6$.

We note the following relation in the braid group that holds for every $i+j\le5$: $\tau^j \sigma_i = \sigma_{i+j}\tau^j$.
In particular, we will make use of the relations $\sigma_3\tau = \tau\sigma_2$ and $\sigma_2\sigma_3\tau = \tau\sigma_1\sigma_2$, and of the fact that $\sigma_1$, $\sigma_3$ and $\sigma_5$ commute.

\begin{figure}
\includegraphics[width=0.8\textwidth]{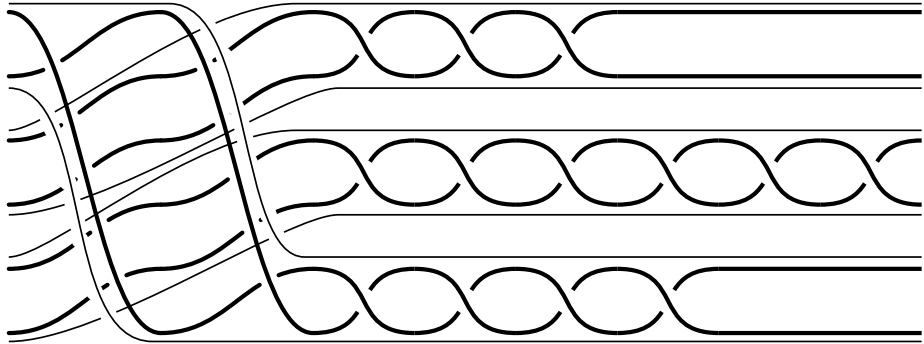}
\caption{The braid $\beta = \tau^2\sigma_1^3\sigma_3^6\sigma_5^4$; strands are numbered from top to bottom, and the braid is to be read from left to right. Note that $\sigma_1$, $\sigma_3$, and $\sigma_5$ commute, hence we can draw them as vertically aligned. The thinner lines exhibit an unknotted solid torus that contains the braid, hence showing that the closure of $\beta$ is a 2-cable of the unknot.}\label{f:braid}
\end{figure}

We claim that $T_{2,19}$ is the closure of the 6-braid $\beta = \tau^2\sigma_1^3\sigma_3^6\sigma_5^4$ (see Figure~\ref{f:braid} above).
In fact, the closure of $\beta$ clearly represents a $2$-cable of the unknot, seen as a 3-braid; therefore, $\widehat\beta$ is a torus knot $T_{2,q}$ for some $q$.
Moreover, since $\beta$ is quasipositive, the slice genus of its closure is determined by the exponent of $\beta$ \cite{R}.
It follows that $g_*(\widehat\beta) = 9$, thus $\widehat\beta = T_{2,19}$, as claimed.

We now exhibit the cobordism, by adding bands corresponding to generators. Indeed, let us write:
\[
\tau^2\sigma_1^3\sigma_3^6\sigma_5^4 = \tau^2(\sigma_1\sigma_3\sigma_5)\sigma_3(\sigma_1\sigma_3\sigma_5)\sigma_3(\sigma_1\sigma_3\sigma_5)\sigma_3\sigma_5;
\]
we can insert bands corresponding to $\sigma_2$ and $\sigma_4$ in each $\sigma_1\sigma_3\sigma_5$ factor, and one extra $\sigma_4$ in the next-to-last position, thus obtaining:
\[
\gamma := \tau^2\cdot\tau\cdot\sigma_3\cdot\tau\cdot\sigma_3\cdot\tau\cdot\sigma_3\cdot\sigma_4\cdot\sigma_5.
\]
We now use the relations $\sigma_3\tau = \tau\sigma_2$ and $\sigma_2\sigma_3\tau = \tau\sigma_1\sigma_2$ mentioned above to write:
\[
\gamma = \tau^3\sigma_3\tau\sigma_3\tau\sigma_3\sigma_4\sigma_5 = \tau^4\sigma_2\sigma_3\tau\sigma_3\sigma_4\sigma_5 = \tau^5\sigma_1\sigma_2\sigma_3\sigma_4\sigma_5 = \tau^6,
\]
and hence $\widehat\gamma = T_{6,6}$.
\end{proof}


We can now cap off the cobordism in $\CP^2$ with six discs, since the Hopf link $T_{6,6}$ is the link at infinity of a degree-6 curve.
Filling the lower boundary component with the Milnor fiber of $T_{2,19}$ (i.e. the pushed-in canonical Seifert surface) we obtain a genus-10 smooth surface $F$ in the homology class $6h\in H_2(\CP^2)$.

The double cover of $\CP^2$ branched over $F$ is a spin 4-manifold $\mathcal{K}$, since the homology class $[F]$ is even, but not divisible by 4~\cite{Nagami}; moreover, the Euler characteristic is 24, since the surface $F$ is of genus 10.
Since $F$ contains the Milnor fiber of $T_{2,19}$, and the double cover of $B^4$ branched over this surface is $P_{19}$~\cite{AK}, we see that $P_{19}$ embeds in $\mathcal{K}$.

\begin{lemma}
The complement of $P_{19}$ in $\mathcal{K}$ admits a handle decomposition with a single $0$-handle and four even-framed $2$-handles.
\end{lemma}

\begin{proof}
$F$ is constructed so that the standard Morse function on $\CP^2$ induces a handle decomposition of $F$ with a single 0-handle, twenty-five 1-handles, and six 2-handles. We now apply work of Akbulut and Kirby~\cite{AK} (see also \cite[Section 6.3]{GS}) to determine the structure of the handle decomposition of the double cover $\Sigma(F)$ of $\CP^2$ branched over $F$. The double cover of $B^4$ branched over a properly embedded surface with a single 0-handle and no 2-handles admits a handle decomposition with a single 0-handle and $k$ 2-handles with even framing, where $k$ is the number of 1-handles of the surface. If we take the double cover of $B^4 \cup 2$-handle branched over such a surface (whose boundary is disjoint from the attaching circle of the 2-handle), the additional 2-handle lifts to two 2-handles in the cover. Now the branched double cover of $B^4$ over the Milnor fiber of $T_{2,19}$ is $P_{19}$, and by the preceding remarks it follows that the branched double cover of the cobordism $C$ constructed above is obtained by attaching only 2-handles from $L_{19} = L(19,18) = \Sigma(T_{2,19})$ to get to $\Sigma(T_{6,6})$. The only potential difficulty is when we take the double cover of $\CP^2 \setminus B^4$ branched over the six disks (the 2-handles of $F$). However, here we can apply a trick from \cite[Section 5]{AK}. Since $F$ is connected, and each disk is glued onto one of the six upper boundary components of $C$, we can connect any pair of disks with an embedded band in $C$. This band can then be ``lifted'' to the 4-handle, connecting the pair of disks and so eliminating one of the 2-handles of $F$. After performing this move five times, we will have isotoped $F$ so that it has a handle decomposition with a single 0-handle and single 2-handle. Then the work of~\cite{AK} implies that the branched double cover admits a handle decomposition without 1- or 3-handles. Since this handle decomposition includes $P_{19}$, the proof is completed by recalling that $\Sigma(F)=\mathcal{K}$ and that $b_2(\mathcal{K}) = 22$. 
\end{proof}

We have not attempted to verify whether $\mathcal{K}$ is diffeomorphic to the $K3$ surface; nor have we attempted to determine the 4-component link in $S^3$ that admits an even-framed surgery to $L_{19}$, whose existence is guaranteed by the argument above.

\begin{claim}
$\emb(L_{19})=4$.
\end{claim}

\begin{proof}
We have shown that $\emb(L_{19}) \leq4$. The proof of the necessary lower bound is by now a familiar argument. If $L_{19}$ embeds in $\#_3 S^2 \times S^2$, then it would split $\#_3 S^2 \times S^2$ into two spin pieces $U$ and $V$, say, with $\partial U =L_{19}$ and $\partial V = \overline{L_{19}}$. Since $\mu(L_{19})=-18 \equiv -2$ (mod 16), we have $\sigma(U)= -2$ and $\sigma(V)=2$. At least one of $U$ or $V$ must have $b_2 \leq3$, and gluing this manifold to either $P_{19}$ or $\overline{P_{19}}$ results in a closed spin 4-manifold $X$ with
\[
 b_2(X)\leq 21  < \frac{10}{8}(16) +2= \frac{10}{8}|\sigma(X)| +2,
\]
which contradicts the 10/8-Theorem.
\end{proof}

At last, we can sum up everything we have obtained to prove Theorem~\ref{t:smallLn}.

\begin{proof}[Proof of Theorem~\ref{t:smallLn}]
In the previous claims we have proven that $\emb(L_2) = 1$ and $\emb(L_{12}) = 11$.
By Theorem~\ref{steps}(2), $\emb(L_n) = n-1$ for $n=2,\dots, 12$.

Analogously, since $\emb(L_{12})= 11$ and $\emb(L_{19}) = 4$, by Theorem~\ref{steps}(2), $\emb(L_n) = 23-n$ for $n=12,\dots, 19$.
\end{proof}

\begin{remark}
We observe here that the computation of the embedding numbers of Theorem~\ref{t:smallLn} can be also done with a case-by-case analysis, without appealing to Theorem~\ref{steps}.
Indeed, one can combine explicit constructions (analogous for those of $L_{12}$, $L_{17}$, and $L_{19}$) either with Rokhlin's theorem (for small $n$) or with the 10/8-Theorem (for large $n$) to achieve the same result.
\end{remark}


\end{subsection}
\end{section}
\begin{section}{Exact calculations for arbitrarily large embedding numbers}\label{exact}

Finally we show how to construct integral and rational homology 3-spheres with arbitrarily large embedding numbers, such that we can give exact calculations provided that we assume the validity of the 11/8-Conjecture.

Let $K_n$ be the spin 4-manifold $K_n = \#_{2n} K3$. Then the intersection form $Q_{K_n}$ of $K_n$ is isomorphic to $4nE_8\oplus 6nH$
(note that our convention is that the $E_8$ form is \emph{negative} definite) and so has signature $-32n$ and rank $44n$. Now $K_n$ admits a handle decomposition without 1-handles or 3-handles.
With this handle decomposition we can perform handle slides so that in the basis for $H_2(K_n)$ corresponding to the cores of the 2-handles, $Q_{K_n}$ is given by $4nE_8\oplus 6nH$.
We can further perform handle slides to obtain a basis such that each of the $n$ copies of $\oplus 6H$ looks like $\left(\begin{array}{cc}Q & I \\ I & O\end{array}\right)$, where all submatrices are $6\times 6$, and 
$Q$ is defined by the following matrix:\\

$$\begin{pmatrix}
     -2 & 1 & 0 & 0 & 0 & 0 \\
    1 & -2 & 1 & 0 & 0 & 0 \\
    0 & 1 & -2 & 0 & 0 & 0 \\
    0 & 0 & 1 & -2 & 1 & 0 \\
    0 & 0 & 0 & 1 & -2 & 1 \\
    0 & 0 & 0 & 0 & 1 & -2 
\end{pmatrix}.$$\\

Observe that this is possible since doubling the linear plumbing of six $-2$-disk bundles over $S^2$ yields $\#_6 S^2 \times S^2$. Now Let $U_n$ be the sub-handlebody of $K_n$ formed by
taking the 0-handle and the 2-handles corresponding to each basis element in the $4n$ $E_8$'s and the first six basis elements in each $\oplus 6H$ block.
Then the intersection form $Q_{U_n}$ of $U_n$ will be $4nE_8 \oplus  nQ$.

One can check that $Q_{U_n}$ is negative definite of rank $38n$, and will have determinant $\pm 7^n$. Hence the boundary $Y_n = \partial U_n$ will be a $\mathbb{Z}/2\mathbb{Z}$-homology sphere,
and therefore has a unique spin structure. Let $V_n$ be the closure of $K_n \setminus U_n$. Then by Novikov additivity and Mayer--Vietoris we get that $V_n$ is positive
definite of rank $6n$, and so $\overline{V_n}$ is negative definite with $\partial \overline{V_n} = Y_n$.
Since $\overline{V_n}$ is a spin 2-handlebody, by doubling we get $\#_{6n} S^2 \times S^2$ and hence $\emb(Y_n) \leq 6n$.

\begin{proposition}
If the $11/8$-Conjecture is true then $\emb(Y_n) = 6n$.
\end{proposition}

\begin{proof}
Fix a natural number $n$. $Y_n$ embeds in $\#_{\emb(Y_n)} S^2 \times S^2$ and splits $\#_{\emb(Y_n)} S^2 \times S^2$ into two spin pieces, $X_1$ and $X_2$.
We can assume that $\partial X_1 = Y_n$ and $\partial X_2 = \overline{Y_n}$. Then let $W_1 = X_1 \cup \overline{U_n}$ and $W_2 = U_n \cup X_2$, where in each case the 4-manifolds
are glued along $Y_n$. Now $W_1$ and $W_2$ are spin 4-manifolds, and so assuming the validity of the 11/8-Conjecture we obtain $b_2(W_1) \geq \frac{11}8|\sigma(W_1)|$
and $b_2(W_2) \geq \frac{11}8|\sigma(W_2)|$. By Novikov additivity and Mayer--Vietoris (and applying what we know about $U_n$) these
become $b_2(X_1) + 38n \geq \frac{11}8(38n+\sigma(X_1))$ and $b_2(X_2) + 38n \geq \frac{11}8(38n-\sigma(X_2))$. Adding these two inequalities
gives $b_2(X_1) + b_2(X_2) +76n \geq \frac{11}8 76n + \frac{11}8(\sigma(X_1) - \sigma(X_2))$. Since $X_1 \cup X_2 = \#_{\emb(Y_n)} S^2 \times S^2$ this simplifies
to $2\emb{(Y_n)} +76n \geq \frac{11}8 76n + \frac{11}8 2\sigma(X_1)$, and after rearranging sides $2\emb{(Y_n)} - \frac{11}8 2\sigma(X_1) \geq \frac38 76n$. 
Finally, since $\sigma(X_1) \geq -\emb{(Y_n)}$, this becomes $\frac{19}4\emb{(Y_n)} \geq \frac38 76n$. Upon simplifying we get that $\emb{(Y_n)} \geq 6n$,
and since we saw previously that $\emb(Y_n) \leq 6n$ it follows that $\emb(Y_n) = 6n$.
\end{proof}

We emphasize that this construction consists of many choices, and each of these choices will affect the resulting manifolds $Y_n$. We can apply a similar argument to compute embedding numbers for integral homology spheres by splitting the intersection form of $\#_{8n} K3$.
The intersection form is $16n E_8 \oplus 24n H$, which is isomorphic to $19n E_8 \oplus -3n E_8$. Applying the same technique as above
splits $\#_{8n} K3$ along an integral homology sphere $Z_n$ that will have embedding number $24n$. Note that this splitting implies that $Z_n$ bounds two simply connected spin 4-manifolds, one of which has intersection form $16nE_8$ and the other $3nE_8$. In particular, the $Z_n$ bound spin, simply connected, negative definite 4-manifolds that have different $b_2$, answering Question 5.2 in~\cite{Tange}.

Our technique has a similar flavor to an argument of Stong~\cite{Stong}, who proved that if the intersection form $Q_X$ of a simply connected closed 4-manifold $X$
decomposes as the direct sum of unimodular forms $U_1 \oplus U_2$, then $X$ can be smoothly decomposed into two simply connected pieces $X_1$ and $X_2$ with $Q_{X_1} = U_1$ and $Q_{X_2} = U_2$. (Note that this is insufficient for our argument; we need that each piece is a 2-handlebody.)
Stong's theorem is a strengthening of a result of Freedman and Taylor~\cite{FT}, which only guarantees that $H_1(X_i) = 0$ rather than that the $X_i$ are simply connected.
Stong's splitting theorem depends on the following structure theorem.

\begin{theorem}[Stong~\cite{Stong}]\label{structure}
Let $X$ denote a simply connected closed smooth $4$-manifold. Then $X$ admits a handlebody decomposition $\mathcal{H}$ with $2$-handles $\{ H_1, \cdots, H_m\}$ such that the following holds.
\begin{enumerate}
\item The attaching circles for the handles $H_1, \cdots, H_r$ represent a free basis for $\pi_1(X^{(1)})$ (where $X^{(1)}$ denotes the union of the $0$-handle and the $1$-handles of $X$),
and the attaching circles for the other $2$-handles are null-homotopic in $\pi_1(X^{(1)})$.
\item The belt spheres of the handles $H_{r+1}, \cdots, H_{r+s}$ represent a free basis for $\pi_1(\overline{X}^{(3)})$ (where $\overline{X}^{(3)}$ denotes the union of the $4$-handle and the $3$-handles of $X$),
and the belt spheres for the other $2$-handles are null-homotopic in $\pi_1(\overline{X}^{(3)})$.
\end{enumerate}
\end{theorem}

Finally, we use this structure theorem to prove a result unrelated to the main theme of the paper, but perhaps interesting in its own right. When the 4-manifold $X$ is non-spin, the result below follows directly from the work of Stong~\cite{Stong} mentioned in the paragraph preceding Theorem~\ref{structure}, and therefore the most interesting case is when $X$ is spin.

\begin{theorem}\label{definite splitting}
Any simply connected $4$-manifold $X$ can be decomposed as $X = X_1 \cup X_2$, where $X_1$ and $X_2$ are simply connected $4$-manifolds that are positive definite and negative definite, respectively, glued along a rational homology sphere.
\end{theorem}

\begin{proof}
Give $X$ a handlebody decomposition as in Theorem \ref{structure} (and we use the notation of Theorem \ref{structure} as well).
Then the cores of the 2-handles $H_{r+s+1}, \dots, H_m$ represent a free basis for $H_2(X) = \mathbb{Z}^{m-r-s}$.
If $X$ is definite, then the theorem is trivial with one of the $X_i$ empty. Otherwise $X$ is indefinite and by the classification of indefinite unimodular forms we have that
$Q_X$ is isomorphic to either $aE_8 \oplus bH$ or $a\langle 1 \rangle \oplus b\langle -1 \rangle$ for some $a$ and $b$, depending on whether $X$ is spin or non-spin. In the non-spin case we slide
and reorder $H_{r+s+1}, \cdots, H_m$
so that the first $a$ handles represent $a\langle 1 \rangle$ and the rest represent $b\langle -1 \rangle$ (where $a+b = m-r-s$). Then we define $X_1$ to be $X^{(1)} \cup H_1 \cup \cdots \cup H_r \cup H_{r+s+1} \cup \cdots \cup H_{r+s+a}$,
and let $X_2$ denote the remaining handles. It follows that the $X_i$ are simply connected and definite as required.

In the spin case we make a similar argument. We have $Q_X = aE_8 \oplus bH$, and we can assume $a$ is nonnegative if we allow reversing the orientation of $X$.
Next we use the fact that $bH$ can be represented by $A=\left(\begin{array}{cc}Q & I \\ I & O\end{array}\right)$, where all submatrices are $b\times b$, and 
$Q$ is negative definite (as we described earlier, this follows since doubling the linear plumbing of $b$ $-2$-spheres results in $\#_b S^2 \times S^2$).
Then we slide and reorder our 2-handles, so that $H_{r+s+1}, \dots, H_{r+s+8a +b}$ represent the first $8a+b$ elements in $aE_8 \oplus A$. 
Then $X_1 := X^{(1)} \cup H_1 \cup \cdots \cup H_r \cup H_{r+s+1} \cup \cdots \cup H_{r+s+8a+b}$ is simply connected and negative definite,
and its complement $X_2$ consisting of the remaining handles will be simply connected and positive definite.

In the non-spin case we have that $\det(X_i)=\pm1$, and so we are actually splitting along an integral homology sphere. In the spin case we still have that $\det(X_i) \neq 0$, and so $\partial X_i$ is a rational homology sphere as required.
\end{proof}

\end{section}

\bibliographystyle{amsalpha}
\bibliography{embedding.bib}

\end{document}